\documentclass[]{amsart}

\usepackage{amsfonts}
\usepackage{amscd}
\usepackage{epsfig}
\usepackage{epic,eepic}
\usepackage{graphicx}
\usepackage[all]{xy}
\usepackage{xypic}
\usepackage{psfrag}
\usepackage{amsmath}
\usepackage{amsthm}
\usepackage{setspace}
\usepackage{url}
\usepackage{here}

\theoremstyle{plain}
\newtheorem{theorem}{Theorem}[section]
\newtheorem{lemma}[theorem]{Lemma}
\newtheorem{proposition}[theorem]{Proposition}
\newtheorem{corollary}[theorem]{Corollary}

\newtheorem{conjecture}[theorem]{Conjecture}

\theoremstyle{definition}
\newtheorem{remark}[theorem]{Remark}

\newtheorem{example}[theorem]{Example}

\title{Square-tiled surfaces and rigid curves on moduli spaces}

\author{Dawei Chen}

\address{University of Illinois at Chicago, Department of Mathematics, Statistics and Computer Science, Chicago, IL 60607}

\email{dwchen@math.uic.edu}

\begin{document}
\bibliographystyle{plain}
\maketitle

\newcommand{\Ag}{\overline{\mathcal A}_{g}}
\newcommand{\Mg}{\overline{\mathcal M}_{g}}
\newcommand{\Hg}{\overline{H}_g}
\newcommand{\Mog}{\overline{M}_{0,2g+2}}
\newcommand{\Mogs}{\widetilde{M}_{0,2g+2}}
\newcommand{\Mon}{\overline{M}_{0,n}}
\newcommand{\Mons}{\widetilde{M}_{0,n}}
\newcommand{\Eff}{\overline{\mbox{Eff}}}
\newcommand{\NE}{\overline{\mbox{NE}}_1}
\newcommand{\T}{\mathcal T_{d,\mu}}
\newcommand{\Ti}{\mathcal T_{d,\mu,i}}
\newcommand{\Mone}{\overline{\mathcal M}_{1,1}}
\newcommand{\Po}{\mathbb P^1}
\newcommand{\HH}{\mathcal H}
\newcommand{\OO}{\mathcal O}
\newcommand{\RR}{\mathbb R}
\newcommand{\QQ}{\mathbb Q}
\newcommand{\CC}{\mathbb C}
\newcommand{\ZZ}{\mathbb Z}
\newcommand{\AB}{\beta^{-1}\alpha^{-1}\beta\alpha}
\newcommand{\Mono}{\overline{M}_{0,n+1}}
\newcommand{\Monos}{\widetilde{M}_{0,n+1}}

\begin{abstract}
We study the algebro-geometric aspect of Teichm\"{u}ller curves parameterizing square-tiled surfaces with two applications: 

$(a)$ there exist infinitely many rigid curves on the moduli space of hyperelliptic curves, they span the same extremal 
ray of the cone of moving curves and their union is Zariski dense, hence they yield infinitely many rigid curves with the same properties on the moduli 
space of stable $n$-pointed rational curves for even $n$; 

$(b)$ the limit of slopes of Teichm\"{u}ller curves and the sum of Lyapunov exponents of the Hodge bundle determine each other, by which we 
can have a better understanding for the cone of effective divisors on the moduli space of curves. 
\end{abstract}

\tableofcontents

\section{Introduction}
Let $\mu = (m_1,\ldots, m_k)$ be a partition of $2g-2$ for $g \geq 2$. The moduli space $\HH(\mu)$ of Abelian differentials parameterizes pairs $(C,\omega)$, 
where $C$ is a smooth genus $g$ curve, $\omega$ is a holomorphic 1-form and $(\omega) = m_1p_1 + \cdots + m_kp_k$ for distinct points 
$p_1,\ldots, p_k$ on $C$. The space $\HH(\mu)$  is a complex orbifold of dimension $2g-1+k$ and the period map yields its local coordinates \cite{K}. It may have up to three connected components \cite{KZ}, corresponding to hyperelliptic, odd or even spin structures. 

Consider a degree $d$ connected cover $\pi: C\rightarrow T$ from a genus $g$ curve $C$ to the standard torus $T$ with a unique branch point $q$, such that 
$\pi^{-1}(q) = (m_1 + 1)p_1 + \cdots + (m_k + 1)p_k + p_{k+1} +\cdots + p_l$. Then $C$ admits a holomorphic 1-form $\omega = \pi^{-1}(dz)$ and $(\omega) = m_1p_1 + \cdots + m_kp_k$. 
It is known \cite[Lemma 3.1]{EO} that such a pair $(C, \omega)$ has integer coordinates under the period map. Varying the complex structure of $T$,  
we obtain a Teichm\"{u}ller curve $\T$ passing through $(C,\omega)$, which is invariant under the natural SL$(2,\RR)$ action on $\HH(\mu)$. One can regard $\T$ 
as the 1-dimensional Hurwitz space parameterizing degree $d$, genus $g$ connected covers of elliptic curves with a unique branch point $q$ and the ramification class $\mu$. 
Abuse our notation and still use $\T$ to denote the compactification of this Hurwitz space in the sense of admissible covers \cite[3.G]{HM2}. The boundary points of $\T$ parameterize
admissible covers of rational nodal curves. We call them cusps of $\T$. Note that $\T$ may be reducible. There is a monodromy criterion \cite[Theorem 1.18]{C} to distinguish its irreducible components, which correspond to the orbits of the SL$(2,\ZZ)$ action. Let $n_{d,\mu}$ be the number of irreducible components of $\T$ and label these components as $\Ti$, $1\leq i \leq n_{d,\mu}$. If $\Ti$ is contained in the hyperelliptic component $\HH^{hyp}(\mu)$ of $\HH(\mu)$, we denote it by $\Ti^{hyp}.$ 
  
Our motivation is to use $\T$ to study the birational geometry of moduli spaces of stable pointed rational curves and stable genus $g$ curves. 
Let $\overline{\mathcal M}_{g,n}$ denote the moduli space of stable $n$-pointed genus $g$ curves. In particular, let $\Mon$ ($\Mons$) denote the moduli space of stable (unordered) $n$-pointed rational curves. There are two natural morphisms as follows: 
$$\xymatrix{
\T \ar[r]^{h} \ar[d]^{e}  &  \Mg  \\
 \Mone             & }
$$ 
The map $h$ sends a cover to the stable limit of its domain curve. The map $e$ sends a cover to its target elliptic curve 
marked at the unique branch point. Moreover, $e$ is finite of degree $N_{d,\mu}$, where $N_{d,\mu}$ is the number of non-isomorphic 
such covers over a fixed elliptic curve. Let $N_{d,\mu,i}$ denote the degree of $e$ restricted to the component $\Ti$, $1\leq i \leq n_{d,\mu}.$

Let $\Hg\subset \Mg$ denote the closure of the locus of smooth hyperelliptic curves. One can regard $\Hg$ as 
the Hurwitz space parameterizing genus $g$ admissible double covers of rational curves. Such a cover uniquely corresponds to 
a stable $(2g+2)$-pointed rational curve, by marking the branch points of the cover. Thus $\Hg$ can be further identified as $\Mogs$. There is a finite morphism 
$$f: \Mon\rightarrow \Mons$$ 
forgetting the order of the marked points. 

Consider two special partitions $\mu = (2g-2)$ or $(g-1, g-1)$. The corresponding moduli spaces $\HH(2g-2)$ and $\HH(g-1,g-1)$ both have a hyperelliptic component $\HH^{hyp}(2g-2)$ and $\HH^{hyp}(g-1,g-1)$, respectively. Note that $h$ maps the Teichm\"{u}ller curves $\Ti^{hyp}$ to $\Hg$ for $\mu = (2g-2)$ or $(g-1, g-1)$. 

\begin{theorem}[Density]
\label{dense}
The union of $h(\Ti^{hyp})$ over all $d, i$ is Zariski dense in $\Hg\cong \Mogs$ for $\mu = (2g-2)$ or $(g-1, g-1)$. Its pre-image in $\Mog$ is also Zariski dense.
\end{theorem}

For a projective variety $X$, let $\NE (X)$ denote its Mori cone of effective curves. Fulton conjectured that $\NE (\Mon)$ ($\NE (\Mons)$) is generated by vital curves. 
By a vital curve, we mean an irreducible component of the 1-dimensional locus in $\Mon$ ($\Mons$) parameterizing pointed rational curves with at least $n-3$ components. 
Fulton's conjecture has been verified for $n\leq 7$ for $\Mon$ \cite{KM} and $n\leq 24$ for $\Mons$ \cite{G}. 

For a map $f: C\rightarrow X$ from an irreducible curve $C$ to a variety $X$, we call $f$ rigid if there does not exist a non-isotrivial family $f_t$ of maps to $X$ such that $f=f_0$. 
If there is no confusion about the map, we also call $C$ rigid on $X$. Suppose the class of an effective curve $R$ generates an extremal ray of $\NE (\Mon)$. Let $f: C\rightarrow R$ be a finite morphism from a curve $C$ to $R$. Keel and McKernan \cite{KM} proved that if $R\cap M_{0,n}\neq \emptyset$, then $f$ must be rigid, cf. \cite[$\S$ 8]{CT} for a precise statement. Hence, it is natural to study rigid curves on $\Mon$ intersecting its interior, which may provide (hypothetical) counterexamples for Fulton's conjecture. Among the few known rigid curves, Castravet and Tevelev studied exceptional loci on $\Mon$ using hypergraph curves and their record is a rigid curve on $\overline{M}_{0,12}$ \cite[Theorem 7.8]{CT}. Koll\'{a}r came up with a series of potential rigid curves (unpublished), whose construction relies on rigid line configurations in $\mathbb P^2$. 

\begin{theorem}[Rigidity]
\label{rigid}
For $g\geq 2$ and $\mu = (2g-2)$ or $(g-1, g-1)$, $\Ti^{hyp}$ is rigid on $\Hg \cong \Mogs $. 
Pull back $\Ti^{hyp}$ via $f: \Mog\rightarrow \Mogs$ for all $d, i$. If $g\geq 3$, infinitely many of the pullbacks are still rigid on $\Mog$.    
\end{theorem}

\begin{corollary}
\label{rigideven}
For even $n\geq 6$ (resp. $\geq 8$), there exist infinitely many rigid curves on $\Mons$ (resp. $\Mon$). The union of their images is Zariski dense in $\Mons$ (resp. $\Mon$). 
\end{corollary}

Let $\phi: \Mono\rightarrow\Mon$ be the morphism forgetting a marked point. For a rigid curve $R$ on $\Mon$, 
$\phi^{-1}(R)$ is the universal curve over $R$, so it is a ruled surface with $n$ sections. These sections are rigid curves 
on $\Mono$. 

\begin{corollary}
\label{rigidodd}
For even $n\geq 8$, there exist infinitely many rigid curves on $\Mono$. The union of their images is Zariski dense in the boundary components $\bigcup_{ |I| = 2}D_I$.  
\end{corollary}

The rational Picard group of $\Hg$ is generated by the boundary components $\Xi_0,\ldots, \Xi_{[(g-1)/2]}$ and $\Theta_1, \ldots, \Theta_{[g/2]}$ \cite[6.C]{HM2}, 
where a general point of $\Xi_i$ parameterizes a double cover of $\Po \cup \Po$ branched at $2i+2$ points in one component and $2g-2i$ in the other, 
and a general point of $\Theta_i$ parameterizes a double cover of $\Po\cup \Po$ branched at $2i+1$ points in one component and $2g-2i+1$ in the other. 
Let $D_k$ be the boundary component of $\Mogs$ whose general point parameterizes a $k$-pointed $\Po$ union a $(2g+2-k)$-pointed $\Po$, $2\leq k\leq g+1$. 
The natural isomorphism $\Hg\cong \Mogs$ induces the identification $\Xi_i = D_{2i+2}$ and $\Theta_i = D_{2i+1}$. 

\begin{theorem}[Extremity]
\label{extremal}
For $\mu = (2g-2)$ or $(g-1, g-1)$, the image of $\Ti^{hyp}$ in $\Hg$ does not intersect the boundary component $D_k$ for any $k > 2$. The numerical class of $\Ti^{hyp}$ spans 
an extremal ray of the cone of moving curves on $\Hg\cong \Mogs$, which is dual to the face $\langle D_3, \ldots, D_{g+1}\rangle $ of the effective cone. 
\end{theorem}

Hassett \cite{H} studied moduli spaces of weighted pointed stable curves. The moduli space $\Mon$ can be regarded as parameterizing pointed stable rational curves with weight 1 on each marked point. Let $\mathcal A(i) = \{ \frac{1}{i}, \ldots, \frac{1}{i}  \}$ be the symmetric weight that assigns $\frac{1}{i}$ to each marked point, $2 \leq i \leq [\frac{n-1}{2}].$ The morphism 
$\rho_i: \Mon \rightarrow \overline{M}_{0,\mathcal A(i)}$ contracts all boundary divisors $D_{I}$ satisfying $ 2< |I| \leq i$. For a rigid Teichm\"{u}ller curve on $\Mon$ as above, its image 
under $\rho_i$ remains rigid on $\overline{M}_{0,\mathcal A(i)}$. 

\begin{corollary}
\label{descend}
The infinitely many Teichm\"{u}ller curves on $\Mon$ descend to infinitely many rigid curves on $\overline{M}_{0,\mathcal A(i)}$ for each even $n\geq 8$ and $2 \leq i \leq [\frac{n-1}{2}]$. 
The union of their images is Zariski dense in $\overline{M}_{0,\mathcal A(i)}$.
\end{corollary}

For a curve $C$ mapped to $\Mg$ by $h$, define its slope 
$$s(C) = \frac{\mbox{deg}\  h^{*}\delta}{\mbox{deg}\  h^{*}\lambda},$$ 
where $\delta$ is the total boundary class of $\Mg$ and $\lambda$ is the first Chern class of the Hodge bundle. 

\begin{corollary}
\label{slope} 
For $\mu = (2g-2)$ or $(g-1, g-1)$, the slope of $\Ti^{hyp}$ is equal to $8 + \frac{4}{g}$ for all $d, i$. 
\end{corollary}

The slope $s(\T)$ is determined by the quotient of two summations $M_{d,\mu}$ and $N_{d,\mu}$ \cite[Theorem 1.15]{C}. 
Understanding the asymptotic behavior of $s(\T)$ is crucial in a number of applications, e.g. it can provide an estimate for the cone of effective divisors on $\Mg$. 
Since $\HH(\mu)$ may have up to three connected components due to hyperelliptic, odd or even spin structures, we use $\T^{hyp}$, $\T^{odd}$ or $\T^{even}$ to denote 
the parts of $\T$ contained in each component, respectively.  

For a connected stratum $\HH(\mu)$, define the limit (if exists) of slopes of $\T$ as
$$s_{\mu} = \lim\limits_{d\to\infty}s(\T).$$ 
Let $\lambda_1 > \cdots > \lambda_g > 0$ denote the positive Lyapunov exponents associated to the Hodge bundle with respect to the Teichm\"{u}ller geodesic flow diag$(e^t, e^{-t})$ on $\HH(\mu)$. Roughly speaking, these numbers measure 
the growth rate of the length of a vector in the bundle under parallel transport along the 
flow, cf. \cite{K} for an introduction on Lyapunov exponents. Let $L_{\mu}$ be the sum 
$$L_{\mu} = \lambda_1 + \cdots + \lambda_g.$$ 
Use $c_{\mu}$ to denote the Siegel-Veech area constant of $\HH(\mu)$, which satisfies the relation 
$$c_{\mu} = \frac{\pi^2}{3}c_{area}$$ 
in the context of \cite{EKZ}. 
Further define $$\kappa_{\mu} = \frac{1}{12}\Big(2g-2+\sum_{i=1}^{k}\frac{m_i}{m_i+1}\Big), $$ 
which is determined by $\mu = (m_1,\ldots, m_k)$. If $\HH(\mu)$ has more than one components, the above quantities $s, L, c$ can be defined in the same way for each component. We distinguish them by adding subscripts $hyp$, $odd$ or $even$, respectively. 
One of the main results \cite{EKZ} is the relation 
$$L_{\mu} = \kappa_{\mu} + c_{\mu}. $$
The key observation is that the quotient $M_{d,\mu}/N_{d,\mu}$ approaches the Siegel-Veech constant $c_{\mu}$ for large $d$, hence we obtain a relation among the limit of slopes, 
the sum of Lyapunov exponents and the Siegel-Veech constant as follows. 

\begin{theorem}[Slope] 
\label{slc}
For a connected stratum $\HH(\mu)$, we have  
$$s_{\mu} = \frac{12 c_{\mu}}{L_{\mu}} = 12 - \frac{12\kappa_{\mu}}{L_{\mu}}. $$
If $\HH(\mu)$ has more than one connected components, the same formula holds for each component. 
\end{theorem}

\begin{corollary}
\label{slc-hyp}
For the hyperelliptic component $\HH^{hyp}(2g-2)$, we have 
$$c^{hyp}_{(2g-2)} = \frac{g(2g+1)}{3(2g-1)}, \  L^{hyp}_{(2g-2)} = \frac{g^2}{2g-1}. $$
For the hyperelliptic component $\HH^{hyp}(g-1,g-1)$, we have 
$$c^{hyp}_{(g-1,g-1)} = \frac{(g+1)(2g+1)}{6g}, \  L^{hyp}_{(g-1,g-1)} = \frac{g+1}{2}. $$
\end{corollary}

The above data match with the results in \cite{EKZ}. 

The slopes of $\T$ can be applied to study the cone of effective divisors on $\Mg$. For an effective divisor $D = a\lambda - b\delta$ on $\Mg$ for $a,b > 0$, define its slope 
$$s(D) = \frac{a}{b}.$$ 
There is neither known effective divisor on $\Mg$ with slope $\leq 6$, nor known lower bound better than $O(1/g)$ for slopes of effective divisors as $g$ approaches infinity. Note that an effective divisor $D$ cannot contain all Teichm\"{u}ller curves, if their union is Zariski dense in $\Mg$. Hence, the limit of slopes of these Teichm\"{u}ller curves provides a lower bound 
for the slope of effective divisors. Given an expectation by Eskin and Zorich on asymptotics of Siegel-Veech constants, the best slope growth of these Teichm\"{u}ller curves turns out to be $576/5g$, which coincides with that of Harris-Morrison's moving curves \cite{HM1}. This coincidence is amusing in that Harris-Morrison's curves are moving while the Teichm\"{u}ller curves
are rigid. 

\begin{conjecture}
There exist effective divisors on $\Mg$ whose slopes can be arbitrarily close to $576/5g$ as $g$ approaches infinity. 
\end{conjecture}

See the end of section 3 for more details on this slope problem. 

This paper is organized as follows. In section 2, we study $\T$ for $\mu = (2g-2)$ or $(g-1,g-1)$. In section 3, we interpret the covers as square-tiled surfaces and study the relation between slopes of $\T$ and Lyapunov exponents. In section 4, we analyze a few examples of square-tiled surfaces with a unique zero. In appendix A, we include an explanation for the relation between 
the limit of slopes of arithmetic Teichm\"{u}ller curves and the Siegel-Veech area constant of the corresponding stratum of Abelian differentials. In appendix B, we list the limit of slopes of $\T$ in each stratum for small genus. Throughout the paper, we work over the complex number field $\CC$. A divisor stands for a $\QQ$-Cartier divisor. We use cusp to denote an intersection point of a Teichm\"{u}ller curve with the boundary of the moduli space. When we refer a torus covering to a square-tiled surface, we emphasize it as a Riemann surface.  \\

\noindent {\bf Acknowledgements.} Section 2 was motivated by conversations with Valery Alexeev. Section 3 and appendix A were written under the help of Alex Eskin. I am very grateful
to both of them for allowing me to use results from related discussions. I sincerely thank Anton Zorich for running a program to compute the limits of slopes for small genus, which are included in appendix B. I also want to express my gratitude to Curt McMullen and Martin M\"{o}ller for explaining to me the rigidity of Teichm\"{u}ller curves.
During the preparation of this work, I had correspondences with many people, including Matt Bainbridge, Izzet Coskun, Lawrence Ein, Angela Gibney, J\'{a}nos Koll\'{a}r, James McKernan, Ian Morrison, Ronen Mukamel, Thomas Peternell and Christian Schnell, among others. I thank all of them for help and encouragement. The results were first announced at KIAS ``Workshop on Moduli and Birational Geometry'', Seoul, Dec 2009. I would like to thank Yongnam Lee for invitation and hospitality. 

\section{Density, Rigidity and Extremity}
In this section, we will prove Theorems~\ref{dense}, \ref{rigid} and \ref{extremal}. 

\begin{proof}[Proof of Theorem~\ref{dense}]
The complex dimension of $\HH(2g-2)$ equals $2g$. Take a symplectic basis $\gamma_{1},\ldots,\gamma_{2g}$ of a genus $g$ Riemann surface $C$.
The period map $\Phi:(C,\omega)\rightarrow \mathbb C^{2g}$ is given by 
$$ \Phi(C,\omega)=\Big(\int_{\gamma_{1}}\omega,\ldots,\int_{\gamma_{2g}}\omega\Big),$$
which provides a local coordinate chart for $\HH(2g-2)$ \cite{K}. 

Take the standard torus $T$ by gluing the parallel sides of $[0,1]^2$. Consider the local coordinates $\Phi(C,\omega)=(\phi_{1},\ldots, \phi_{2g})\in \mathbb C^{2g}\cong \RR^{4g}$. 
By \cite[Lemma 3.1]{EO}, we have $\phi_i\in \mathbb Z^2$ for $i=1,\ldots, 2g$ if and only if the following holds: \\
(1) there exists a holomorphic map $f: C\rightarrow T$; \\ 
(2) $\omega = f^{-1}(dz)$; \\ 
(3) $f$ has a unique ramification point at $p$ and $(\omega) = (2g-2)p$; \\
(4) the degree of $f$ is equal to $\frac{\sqrt{-1}}{2}\int_{C}\omega\wedge\overline{\omega}. $ 

This establishes a correspondence between integer points of $\HH(2g-2)$ and genus $g$ covers $C\rightarrow T$ with a unique ramification point. 
It follows that such covers form a Zariski dense subset in $\HH(2g-2)$. Since $\HH^{hyp}(2g-2)$ is a connected component of $\HH(2g-2)$, 
the union of hyperelliptic components $\mathcal T_{d,(2g-2),i}^{hyp}$ of Teichm\"{u}ller curves $\mathcal T_{d,(2g-2)}$ for all $d, i$ forms a Zariski dense subset in $\HH^{hyp}(2g-2)$. 
Since $\HH^{hyp}(2g-2)$ admits a dominant map to $\Hg$, the image of the union is Zariski dense in $\Hg$. Using the isomorphism $\Hg\cong \Mogs$ and the finite morphism $\Mog\rightarrow \Mogs$, the pre-image of the union in $\Mog$ is also Zariski dense. 

For $\HH(g-1,g-1)$, it has dimension $2g+1$. Take a path $\gamma_{2g+1}$ connecting the two zeros $p_1, p_2$ of $\omega$ on $C$. 
The period map $\Phi:(C,\omega)\rightarrow \mathbb C^{2g+1}$ is given by 
$$ \Phi(C,\omega)=\Big(\int_{\gamma_{1}}\omega,\ldots,\int_{\gamma_{2g+1}}\omega\Big).$$
The rest argument is the same as the previous case. 
\end{proof}

A Teichm\"{u}ller curve on $\mathcal M_g$ is an algebraic geodesic with respect to the Kobayashi (equivalently, Teichm\"{u}ller) metric. More precisely, pull back the Hodge bundle from $\mathcal M_g$ to the Teichm\"{u}ller space $\mathcal T_g$ and consider it as a real manifold. There is an SL$(2,\RR)$ action on the Hodge bundle, induced by the natural SL$(2,\RR)$ action on the real and imaginary parts of a holomorphic 1-form. The fibers of the Hodge bundle are stabilized by SO$(2,\RR)$. The induced map $\mathbb H\cong \mbox{SL}(2,\RR) / \mbox{SO}(2,\RR)\rightarrow \mathcal T_g$ is a holomorphic isometry. In some rare occasions, the image of the composite map
$\mathbb H\rightarrow \mathcal T_g\rightarrow \mathcal M_g$ forms an algebraic curve in $\mathcal M_g$ and we call it a Teichm\"{u}ller curve. In general, for $f: C\rightarrow \Mg$, we say that $(C, f)$ is a Teichm\"{u}ller curve if $f$ is a local isometry with respect to the Kobayashi metrics on  domain and range in $\mathcal M_g$.  

The 1-dimensional Hurwitz spaces $\Ti$ parameterizing torus covers with a unique branch point (also called square-tiled surfaces or origamis, as we will see in section 3) are invariant under the SL$(2,\RR)$ action, since the action amounts to varying the defining lattice of the target elliptic curve. Then $h: \Ti \rightarrow \Mg$ is a local isometry and it is called an arithmetic Teichm\"{u}ller curve. The rigidity of $\Ti$ essentially follows from the rigidity of general Teichm\"{u}ller curves. 

\begin{proof}[Proof of Theorem~\ref{rigid}]
McMullen \cite{CTM} and M\"{o}ller \cite{M} both proved that a Teichm\"{u}ller curve is rigid on $\Mg$. In fact, they proved that the rigidity keeps for a finite morphism of a 
Teichm\"{u}ller curve unramified away from cusps, where a cusp corresponds to a singular covering in the Hurwitz space in our setting. A component $\Ti$ of $\T$ is an arithmetic Teichm\"{u}ller curve, hence it is rigid on $\Mg$. If $\Ti^{hyp}$ maps into $\Hg\subset \Mg$, then $\Ti^{hyp}$ is also rigid on $\Hg\cong \Mogs$. 

Consider the finite morphism $f: \Mog\rightarrow \Mogs\cong \Hg$. Let $U \subset H_g$ denote the branch locus of $f$ restricted to the interior of the moduli space. For $g\geq 3$, we claim that infinitely many of $h(\Ti^{hyp})$ do not meet $U$. Then $f^{-1}(h(\Ti^{hyp}))\rightarrow h(\Ti^{hyp})$ is unramified away from cusps, hence the pullback of $\Ti^{hyp}$ via $f$ 
is also rigid on $\Mog$. 

By \cite[Lemma 3.3]{KM}, the codimension of $U$ is at least two in $H_g$ for $2g+2\geq 7$, i.e. $g\geq 3$. Let $U(\mu)$ denote the locus of pairs $(C, \omega)$ in $\HH^{hyp}(\mu)$ for $\mu = (2g-2)$ or $(g-1,g-1)$, where $C$ is parameterized in $U$. If $U$ intersects $h(\Ti^{hyp})$, then $U(\mu)$ has to intersect the SL$(2,\mathbb R)$ invariant orbit 
$\Ti^{hyp}$ in $\HH^{hyp}(\mu)$, which is generated by an integer point under the period map coordinates of $\HH^{hyp}(\mu)$. By dimension count, the SL$(2,\mathbb R)$ orbit of $U(\mu)$ is a proper subspace of $\HH^{hyp}(\mu)$, which must miss infinitely many integer points. Namely, the SL$(2,\mathbb R)$ orbits of these integer points do not meet $U(\mu)$. Consequently, the images of the Teichm\"{u}ller curves $\Ti^{hyp}$ generated by these integer points do not meet $U$ in $H_g$. 
\end{proof}

Now Corollary~\ref{rigideven} follows immediately as a consequence of Theorems~\ref{dense} and ~\ref{rigid}. 

Let $R$ denote a rigid curve on $\Mon$. Let $\phi: \Mono\rightarrow \Mon$ be the morphism forgetting a marked point and stabilizing the rest part of the curve. 
Note that $S = \phi^{-1}(R)$ is a ruled surface over $R$ with $n$ sections. Each section can be regarded as a curve on $\Mono$ contained in its boundary components 
$\bigcup_{|I|=2} D_I$. 

\begin{lemma}
\label{ruled}
The $n$ sections of $S$ are rigid curves on $\Mono$. 
\end{lemma}

\begin{proof}
Let $\Gamma$ denote a section of $S$. Since $\phi (\Gamma) = R$, if $\Gamma$ deforms in a surface $S'$ in $\Mono$, the image $\phi(S')$ must be $R$.
Otherwise $R$ would deform in $\phi(S')$, which contradicts its rigidity. Therefore, we get $S' = S$. But $S$ can be constructed from successive blow-ups of 
$R\times \Po$ and each blow-up decreases the self-intersection of a section passing through the blow-up center. Hence, $\Gamma^{2} < 0$ and $\Gamma$ does not deform in $S$. 
\end{proof}

\begin{proof}[Proof of Corollary~\ref{rigidodd}]
Consider the following diagram: 
$$\xymatrix{
\Mono  \ar[d]^{\phi}  &    \\
 \Mon \ar[r]^{f}             &  \Mons      }
$$ 

Let $R$ denote a Teichm\"{u}ller curve on $\Mons$ such that 
$f^{-1}(R)\rightarrow R$ is unramified away from the cusps of $R$, i.e. the image of $R$ does not meet the branch locus of $f$ in the interior of $\Mons$. 
Then $R$ and $f^{-1}(R)$ are rigid on $\Mons$ and $\Mon$, respectively. Interpret $\phi$ as the morphism forgetting the $i$-th marked point and reordering the other $n$ marked points. 
In other words, $\phi$ is determined by a set bijection $\{1,\ldots,n \} \rightarrow \{1,\ldots,\hat{i},\ldots, n+1 \}$. Abuse our notation and also use $\phi$ to denote this bijection. 
Let $R_{i,\phi(j)}$ be the $j$-th section 
of the universal curve over $f^{-1}(R)$ for such $\phi$, $1\leq j \leq n$. By Lamma~\ref{ruled}, $R_{i,\phi(j)}$ is rigid on $\Mono$ and contained in the boundary component $D_{I}$, where $I = \{i, \phi(j) \}$. Consider all possible $R, \phi, i, j$ and we obtain infinitely many rigid curves $R_{i,\phi(j)}$. Corollary~\ref{rigidodd} now follows as a consequence of Theorems~\ref{dense} and \ref{rigid}. 
\end{proof}

Let us analyze the cusps of $\T$ parameterizing singular admissible covers over a rational nodal curve for $\mu = (2g-2)$ or $(g-1,g-1)$. This can help us understand the intersection of $h(\T)$ with the boundary components of moduli spaces. See \cite[3.G]{HM2} for an introduction on admissible covers. 

\begin{proposition}
\label{stable limit}
If $\pi: C \rightarrow E_0$ parameterized by $\mathcal T_{d,(2g-2)}$ or $\mathcal T_{d,(g-1,g-1)}^{hyp}$ is an admissible cover over a rational nodal curve $E_0$, the stable limit of $C$ is an irreducible nodal curve. 
\end{proposition}

\begin{proof}
Consider the case $\mu = (2g-2)$ first. Let $C_0$ be the irreducible component of $C$ that contains the unique ramification point $p$. If $C_1$ is another irreducible component of $C$, note that $\pi$ restricted to $C_1$ is not ramified away from the nodes of $C$. We claim that $C_1$ is a smooth rational curve and it meets $\overline{C\backslash C_1}$ at two points. 

Suppose the restriction of $\pi$ to $C_1$ is a degree $d_1$ map to $E_0$ and $q$ is the node of $E_0$. Let $(U,V)$ denote the two local branches of $q$ in $E_0$ with 
local coordinates $(u, v)$. If $C_1$ is singular, let $r_1, \ldots, r_m$ be its nodes. Note that $\pi(r_i) = q$ and locally around $r_i$, $\pi$ is given by $(x,y)\rightarrow (u= x^{a_i}, v = y^{a_i})$. 
Let $s$ be an intersection point in $\{C_1\cap\overline{C\backslash C_1}\}$ and $\pi (s) = q$. If a local neighborhood of $s$ on $C_1$ maps to $U$ ($V$) given by $u=x^{b}$ 
($v = x^{b})$, we call $s$ of type $(u, b)$ ($(v, b)$). Starting from $s$ of type $(u, b)$, by the definition of admissible covers, $s$ also belongs to another component $C_1'$ of $C$, where $s$ as a smooth point on $C_1'$ is of type $(v, b)$. Then there exists $s'$ as a smooth point on $C_1'$ such that $\pi (s') = q$ and $s'$ on $C_1'$ is of type 
$(u, b)$. To pair with $s'$, there ought to be some $s''$ as a smooth point on a component $C_1''$ such that $\pi (s'') = q$ and $s''$ on $C_1''$ is of type 
$(v, b)$. Since there are finitely many components, the process has to stop at some stage. Namely, it ends with a point, say $t$ on the starting component $C_1$ such that 
$t$ is of type $(v, b)$, which pairs with $s$ of type $(u, b)$ mapping to $q$. Therefore, the set $\{C_1\cap\overline{C\backslash C_1}\}$ contains even number of points that decompose into 
pairs $(s_1, t_1), \ldots, (s_k, t_k)$. Locally around $(s_j, t_j)$ on $C_1$, $\pi$ is given by $(x,y)\rightarrow (u = s_j^{b_j}, v = t_j^{b_j})$. 

Let $C_1^{\nu}$ be the normalization of $C_1$ and $r_i', r_i''$ be the pre-image of $r_i$. Let $q'$ and $q''$ be the two points on $\Po$ identified as the node $q$ on $E_0$. 
The map $\pi$ induces a degree $d_1$ branched cover $\pi^{\nu}: C_1^{\nu}\rightarrow \Po$ 
with ramification points $r_1', r_1'', \ldots, r_m', r_m''$ and $s_1, t_1, \ldots, s_k, t_k$. 
The ramification order of $r_i'$ and $r_i''$ is equal to $a_i-1$. The ramification order of $s_j$ and $t_j$ 
is equal to $b_j-1$. By Riemann-Roch, we have 
$$2g(C_1^{\nu}) - 2 + 2d_1 = 2 \sum_{i=1}^m (a_i-1) + 2\sum_{j=1}^k (b_j-1),$$ 
that is, 
$$g(C_1^{\nu}) -1 + d_1 =  \sum_{i=1}^m a_i +\sum_{j=1}^k b_j - m - k. $$
Since $\pi^{\nu}$ maps $r_1', \ldots, r_m'$ and $s_1, \ldots, s_k$ to $q'$, we have $\sum_{i=1}^m a_i +\sum_{j=1}^k b_j \leq d_1$. 
Then we get $0 \leq g(C_1^{\nu}) \leq 1 - m - k,$ which implies $m+k \leq 1$. 
But $2k = | C_1\cap\overline{C\backslash C_1} |$ is positive, otherwise $C$ would be disconnected. So the only possibility 
is $g(C_1^{\nu}) = m = 0$ and $k = 1$. This says that $C_1$ is a smooth rational curve and $| C_1\cap\overline{C\backslash C_1} |= 2$.  

Using the same argument, the intersection points $\{C_0\cap\overline{C\backslash C_0}\}$ decompose in pairs. For two points in the same pair, they are connected by 
a chain of rational curves, each of which has the same property as $C_1$. Hence, the whole curve $C$ looks like the following: 
\begin{figure}[H]
    \centering
    \psfrag{C0}{$C_0$}
    \psfrag{s1}{$s_{1}$}
    \psfrag{s2}{$s_{2}$}
    \psfrag{s3}{$s_{3}$}
    \psfrag{t1}{$t_1$}
    \psfrag{t2}{$t_{2}$}
    \psfrag{t3}{$t_{3}$}
    \includegraphics[scale=0.6]{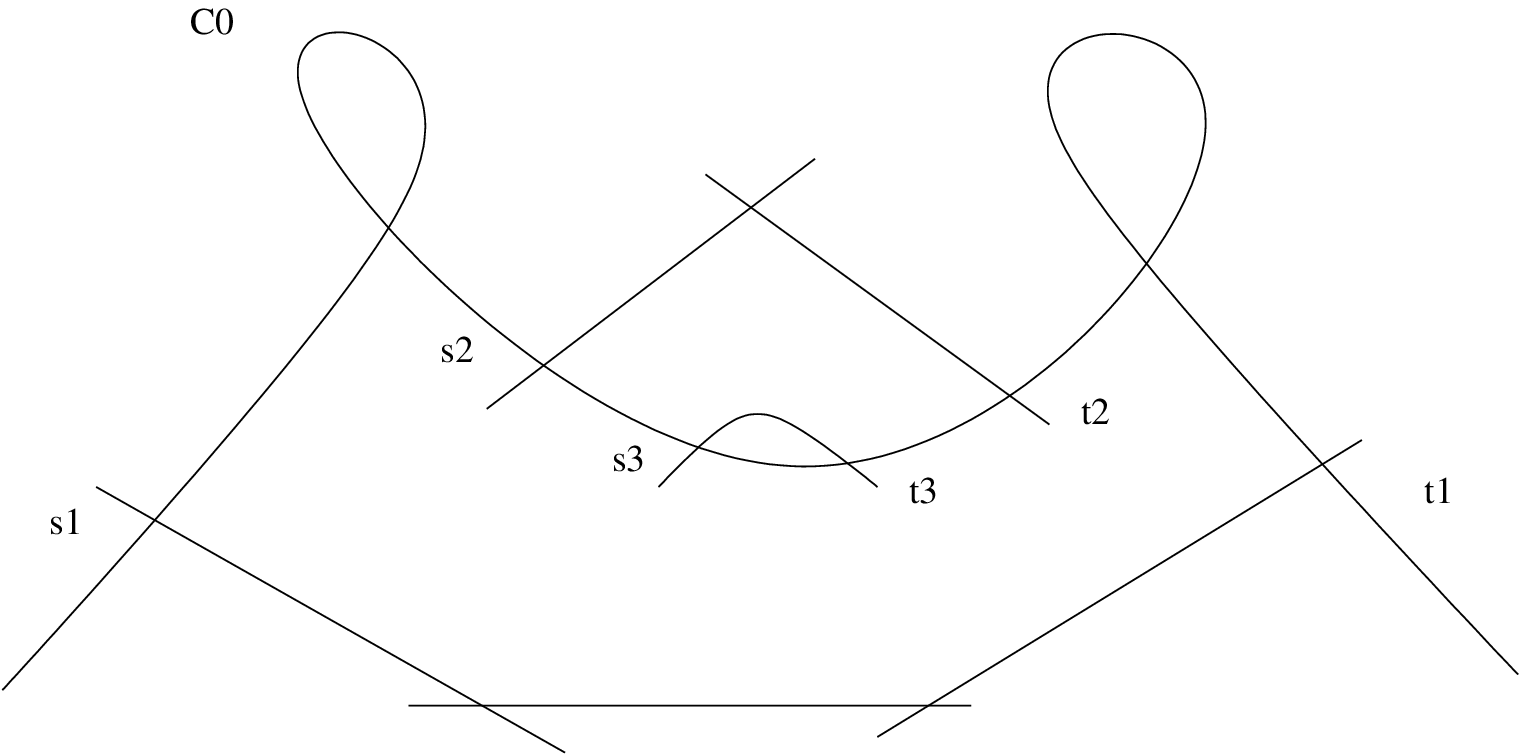}
    \end{figure}
In the above picture, $C_0$ is the component containing the unique ramification point, and the points in $\{C_0\cap\overline{C\backslash C_0}\}$ decompose to 
pairs $(s_i, t_i)$ such that $s_i$ and $t_i$ are linked by a chain of smooth rational curves. Blowing down all the smooth rational components, the stable limit of $C$ is 
an irreducible nodal curve by gluing the points $\{C_0\cap\overline{C\backslash C_0}\}$ in pairs. 

For $\mu = (g-1,g-1)$, by the definition of $\HH^{hyp}(g-1,g-1)$ \cite{KZ}, the two zeros $p_1, p_2$ of the holomorphic 1-form $\omega$ on $C$ are switched by the hyperelliptic involution $\iota$. If $p_1, p_2$ belong to two different irreducible components $C_1$ and $C_2$ of the stable limit of $C$, then $C$ admits a double cover of a rational curve that maps $C_1$ and $C_2$ to different components of the rational curve. Since the target rational curve is fixed by $\iota$, the two zeros $p_1, p_2$ cannot be switched by $\iota$, contradiction. Hence, $p_1, p_2$ are contained in the same component $C_0$. The rest argument is the same as the previous case. 
\end{proof}

\begin{remark}
Proposition~\ref{stable limit} may fail for singular admissible covers in non-hyperelliptic components of $\mathcal T_{d, (g-1,g-1)}$. For instance, let $E$ be an elliptic curve that admits a 
triple cover of $\mathbb P^1$ with three ramification points $p, q, r$ of ramification order two. Glue $E$ with another copy $E'$ at $p=q', q=p'$ and we obtain a reducible curve $C$ with two nodes. The arithmetic genus of $C$ is three. Then $C$ admits a sextic cover of a rational 1-nodal curve $E_0$ such that $p, q$ map to the node of $E_0$ and $r, r'$ are the only two ramification points in $C_{sm}$ with ramification order two. See the picture below: 
\begin{figure}[H]
    \centering
    \psfrag{E}{$E$}
    \psfrag{F}{$E_0$}
    \psfrag{P}{$\mathbb P^1$}
    \psfrag{p}{$p$}
    \psfrag{q}{$q$}
    \psfrag{r}{$r$}
    \psfrag{C}{$C$}
    \psfrag{3:1}{$3:1$}
    \psfrag{6:1}{$6:1$}
    \psfrag{7}{$p=q'$}
    \psfrag{2}{$q=p'$}
    \psfrag{4}{$r'$}
    \psfrag{5}{$E'$}
    \includegraphics[scale=0.6]{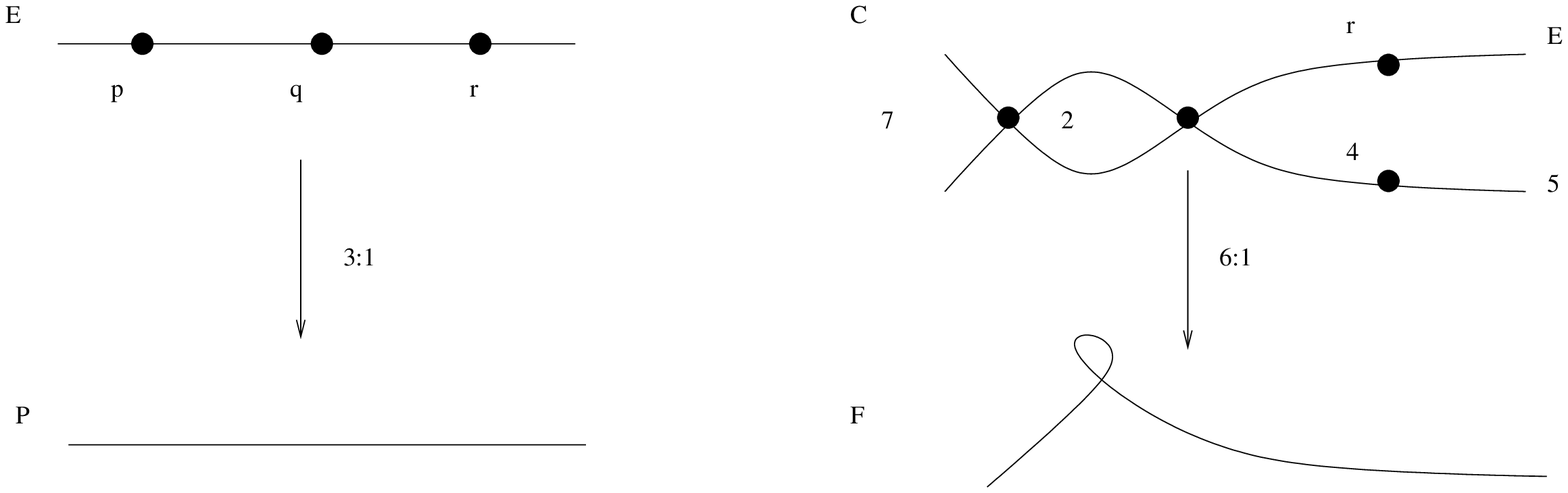}
    \end{figure}
This cover $C\rightarrow E_0$ is a limit of smooth covers in $\mathcal T_{6,(2,2)}$. Note that $C$ itself is stable but reducible. 
The essential difference of this cover from those in $\mathcal T_{d,(g-1,g-1)}^{hyp}$ is that the two zeros $r, r'$ are contained in two distinct stable components. For a cover in $\mathcal T_{d,(2g-2)}$, since the holomorphic 1-form $\omega$ has a unique zero, this issue does not occur. Moreover, though $C$ is reducible in the example, it is not contained in boundary components $\delta_i$ of $\Mg$ for any $i>0$. This holds in general for singular admissible covers in any $\T$ \cite[Proposition 3.1]{C}.  
\end{remark}

Now we prove the extremity of $\Ti^{hyp}$ for $\mu = (2g-2)$ or $(g-1,g-1)$. 

\begin{proof}[Proof of Theorem~\ref{extremal}]
For $\mu = (2g-2)$ or $(g-1,g-1)$, we want to show that $h(\Ti^{hyp})$ does not intersect the boundary component $D_k$ for $k > 2$. Let $\phi: C\rightarrow B$  
be an admissible double cover over a stable $(2g+2)$-pointed rational curve $B$ such that the stable limit of $C$ is the same as that of an admissible cover 
over a rational 1-nodal curve $E_0$ parameterized by $\Ti^{hyp}$. The map $\phi$ is branched at the $2g+2$ marked points of $B$. 
By Proposition~\ref{stable limit}, the stable limit of $C$ is an irreducible nodal curve. Then $C$ consists of 
a smooth component $C_0$ with $m$ pairs of points $(s_i, t_i)$ on $C_0$, where $s_i, t_i$ are linked by a chain of smooth rational curves in $C$. See the picture below: 
\begin{figure}[H]
    \centering
    \psfrag{C0}{$C_0$}
    \psfrag{s1}{$s_{1}$}
    \psfrag{s2}{$s_{2}$}
    \psfrag{t1}{$t_1$}
    \psfrag{t2}{$t_{2}$}
    \includegraphics[scale=0.5]{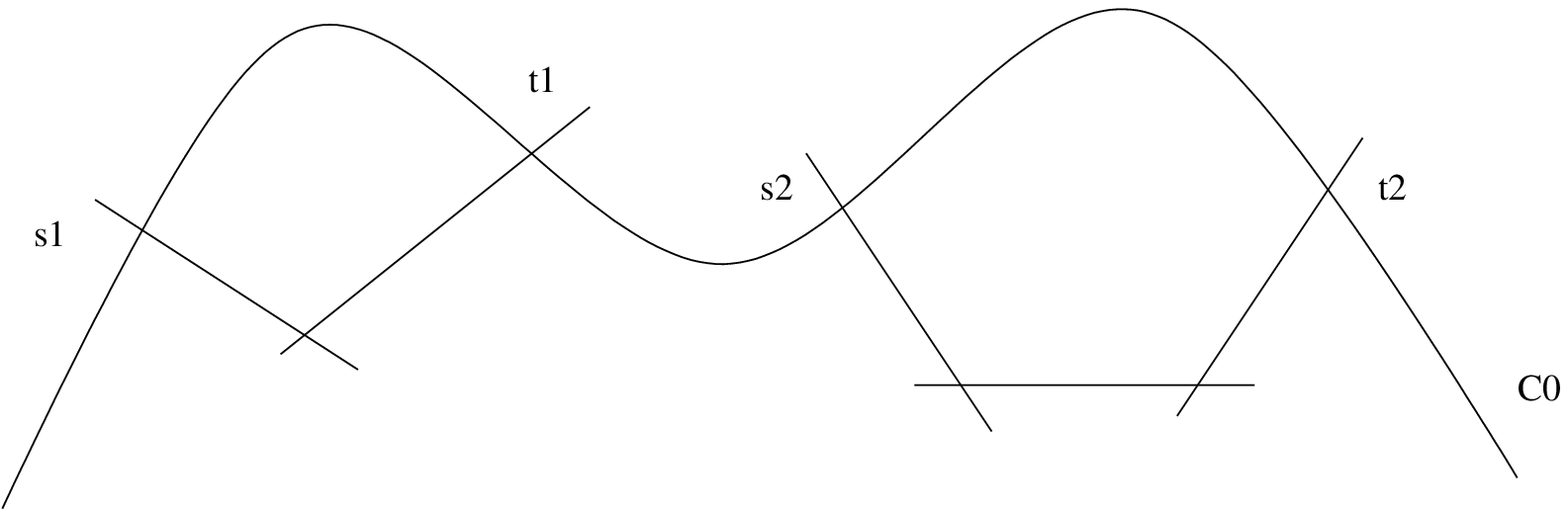}
    \end{figure}
    
Let $B_0$ be the image $\phi (C_0)$ in $B$. Call an irreducible component $B_1$ of $B\backslash B_0$ a tail if its removal does not disconnect $B$. We claim that a tail must intersect $B_0$. 
Suppose $B_1$ is a tail that does not meet $B_0$. Let $C_1$ denote the pre-image $\phi^{-1}(B_1)$. 
By the stability of $B$, $B_1$ contains two marked points and $\phi$ restricted to $C_1$ is a double cover of $B_1$ branched at these two points. So $C_1$ is an irreducible component of a chain of rational curves in $C$.  Let $r$ be the node $B_1\cap \overline{B\backslash B_1}$ and $B_r$ be the unique irreducible component of $B$ that meets $B_1$ at $r$. 
Let $p, q$ be the pre-image of $r$ under $\phi$. Note that $p, q$ are not in $C_0$. By the description of $C$, there exist two different components $C_p$ and $C_q$ of $C$ that meet $C_1$ at $p$ and $q$, respectively. Since $r$ is not a branch point, $\phi$ maps $C_p$ and $C_q$ both to $B_r$ isomorphically. Then $B_r$ does not contain any marked point, which contradicts  the stability of $B$. The following picture illustrates the idea of this argument: 
\begin{figure}[H]
    \centering
    \psfrag{C1}{$C_1$}
    \psfrag{C2}{$C_{p}$}
    \psfrag{C3}{$C_{q}$}
    \psfrag{C0}{$C_0$}
    \psfrag{p}{$p$}
    \psfrag{q}{$q$}
    \psfrag{r}{$r$}
    \psfrag{B0}{$B_0$}
    \psfrag{B1}{$B_1$}
    \psfrag{B2}{$B_r$}
    \psfrag{hi}{$\phi$}
    \includegraphics[scale=0.5]{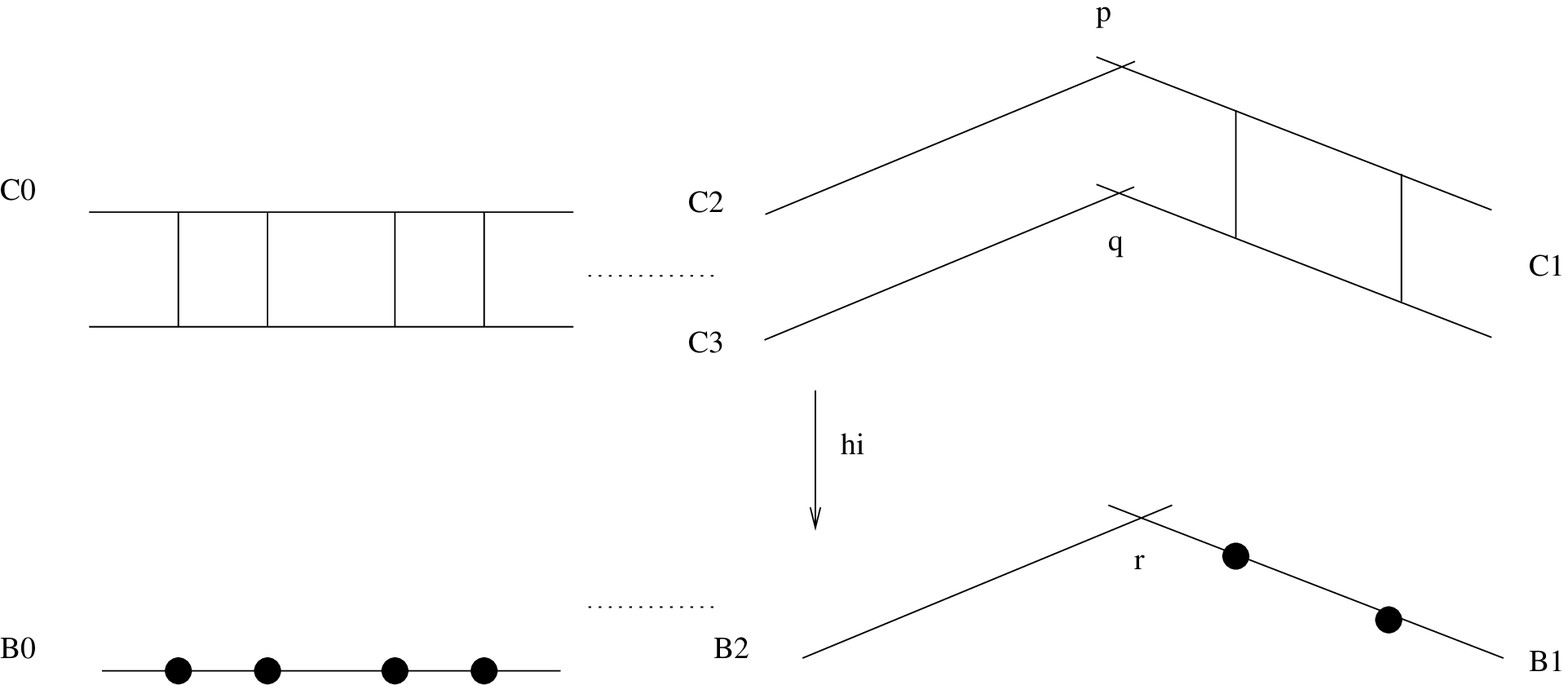}
    \end{figure}

Therefore, $B$ consists of the main component $B_0$ with all rational tails intersecting $B_0$ such that each tail contains two marked points. The pre-image of a tail under $\phi$ is an irreducible rational component of $C$ intersecting $C_0$ at a pair of points $s_i, t_i$. An example of such admissible covers $\phi$ is as follows: 
\begin{figure}[H]
    \centering
    \psfrag{C0}{$C_0$}
    \psfrag{s1}{$s_1$}
    \psfrag{s2}{$s_2$}
    \psfrag{t1}{$t_1$}
    \psfrag{t2}{$t_2$}
    \psfrag{B0}{$B_0$}
    \psfrag{hi}{$\phi$}
    \includegraphics[scale=0.5]{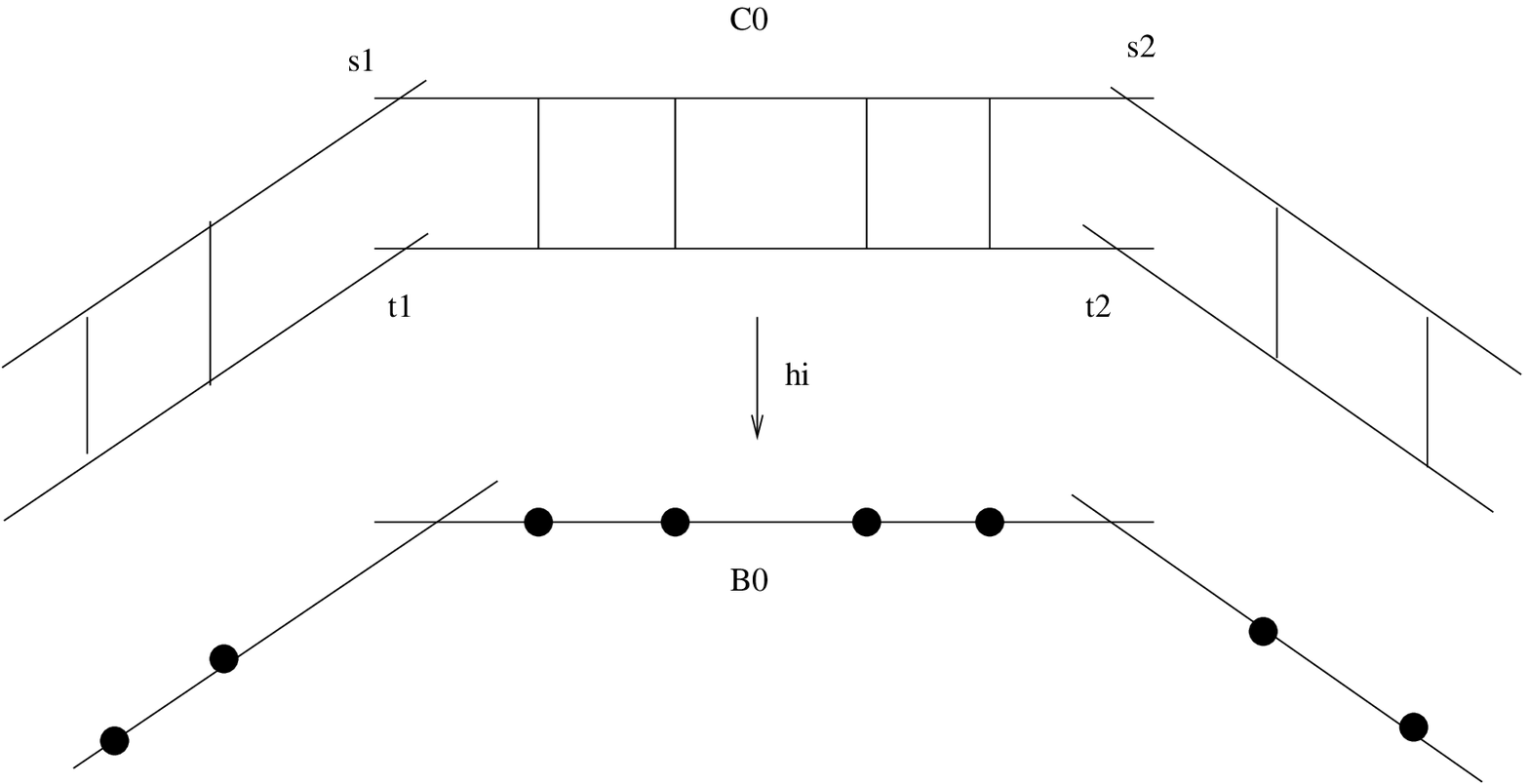}
    \end{figure}

Note that removing any node of $B$, the resulting two connected components have 2 and $2g$ marked points, respectively. Hence, $B$ as a $(2g+2)$-pointed rational curve is not parameterized in the boundary component $D_k$ of $\Mogs$ for any $k > 2$. Then the intersection $h(\Ti^{hyp}) \ldotp D_k$ is zero for any $k > 2$. Keel and McKernan [KM] showed that the effective cone of $\Mogs$ is generated by $D_2, \ldots, D_{g+1}$. The pseudo-effective cone of a projective variety is dual to its cone of moving curves with respect to the natural intersection paring \cite{BDPP}. So the numerical class of $h(\Ti^{hyp})$ spans an extremal ray of the cone of moving curves on $\Mogs$, which is dual to the face $\langle D_2, \ldots, D_{g+1}\rangle$ 
of the effective cone. 
\end{proof}

\begin{remark}
\label{movingrigid}
Take a $(2g+2)$-pointed smooth rational curve and vary a marked point. We get a curve in $\Mog$ whose image $C$ in $\Mogs$ does not intersect any boundary component $D_k$ for $k > 2$. 
So $C$ and the above $h(\Ti^{hyp})$ span the same extremal ray of the cone of moving curves on $\Mogs$. It is surprising in that $\Ti^{hyp}$ is rigid but the deformations of $C$ cover
an open subset of $\Mogs$. Nevertheless, this ray can be expressed as a non-negative linear combination of vital curves. Unfortunately (or fortunately, depending on your attitude) these rigid curves do not provide counterexamples for Fulton's conjecture. Note that the Teichm\"{u}ller curves on $\Mon$ are different from those constructed in \cite{CT} using the exceptional loci of birational contractions of $\Mon$. The reason is because such Teichm\"{u}ller curves and the moving curve $C$ span the same ray, any contraction that blows down a Teichm\"{u}ller curve
must contract $\Mon$ to a lower dimensional target. In particular, it cannot be a birational contraction. 
\end{remark}

The descending property of the rigidity of $\Ti^{hyp}$ and the value of their slopes for $\mu = (2g-2)$ or $(g-1,g-1)$ follow as direct consequences of Theorem~\ref{extremal}. 

\begin{proof}[Proof of Corollary~\ref{descend}]
Let $n=2g+2$. By Theorem~\ref{extremal}, we know $h(\Ti^{hyp})$ does not meet the boundary components $D_j$ for any $j > 2$ in $\Mons$ for even $n\geq 6$. Their pre-images $f^{-1}(h(\Ti^{hyp}))$ in $\Mon$ do not meet the boundary components $D_{J}$ for any $|J| > 2$. Since the morphism $\rho_i: \Mon\rightarrow \overline{M}_{0,\mathcal A(i)}$ only contracts the boundary components $D_{J}$ for $2<|J|\leq i$, the restriction of $\rho_i$ to a local neighborhood of  $f^{-1}(h(\Ti^{hyp}))$ is an isomorphism. Hence, if $f^{-1}(\Ti^{hyp})$ is rigid on $\Mon$, its projection to $\overline{M}_{0,\mathcal A(i)}$ is also rigid. Corollary~\ref{descend} now follows from the combination of Theorems~\ref{dense} and \ref{rigid}. 
\end{proof}

\begin{proof}[Proof of Corollary~\ref{slope}] 
The inclusion $\iota: \Hg\hookrightarrow \Mg$ induces a pull-back map
$$ \iota^{*}: \mbox{Pic}(\Mg)\otimes \mathbb Q \rightarrow \mbox{Pic}(\Hg) \otimes \mathbb Q, $$ 
such that $\iota^{*} (\Delta_0) = 2\sum D_{2i}$ and $\iota^{*}(\Delta_i) = D_{2i+1}/2$ for $i > 0$ \cite[6.C]{HM2}. 
Moreover, we have 
$$ \iota^{*}(\lambda) = \sum_{i=0}^{[(g-1)/2]}\frac{(i+1)(g-i)}{4g+2} D_{2i+2} + \sum_{i = 1}^{[g/2]} \frac{i(g-i)}{4g+2} D_{2i+1}. $$
For $\mu = (2g-2)$ or $(g-1,g-1)$, $\Ti^{hyp}$ maps to $\Hg$ and its image does not intersect $D_k$ for any $k > 2$. Then it has slope 
$$ \frac{h(\Ti^{hyp})\ldotp \delta}{h(\Ti^{hyp})\ldotp \lambda} =  2 \cdot \frac{h(\Ti^{hyp})\ldotp D_2}{h(\Ti^{hyp})\ldotp \iota^{*}(\lambda)} = 2 \cdot \frac{4g+2}{g} = 8 + \frac{4}{g}. $$  
\end{proof}

\begin{remark}
There is a general formula \cite[Theorem 1.15]{C} to compute the slope of such 1-dimensional Hurwitz spaces of covers. But the combinatorics involved in the formula is so 
complicated that the author was only able to calculate the slope for $g = 2, 3$. Here the result implies that for a hyperelliptic component $\Ti^{hyp}$ for $\mu = (2g-2)$ or $(g-1,g-1)$,
we know its slope equal to $8+\frac{4}{g}$ without doing any calculation! Nevertheless, see Example~\ref{53} for a reducible Hurwitz space whose hyperelliptic components have slope $8+\frac{4}{g}$, while the others have a different slope. 
\end{remark}

\section{Slopes, Lyapunov exponents and Siegel-Veech constants}
In this section, we will prove Theorem~\ref{slc}. 

Let $a, b, c$ denote a standard basis of $\pi_1 (E, q)$ for a torus $E$ punctured  at $q$, 
where they satisfy $b^{-1}a^{-1}ba = c$. 
\begin{figure}[H]
    \centering
    \psfrag{E}{$E$}
    \psfrag{q}{$q$}
    \psfrag{a}{$a$}
    \psfrag{b}{$b$}
    \psfrag{c}{$c$}
    \includegraphics[scale=0.5]{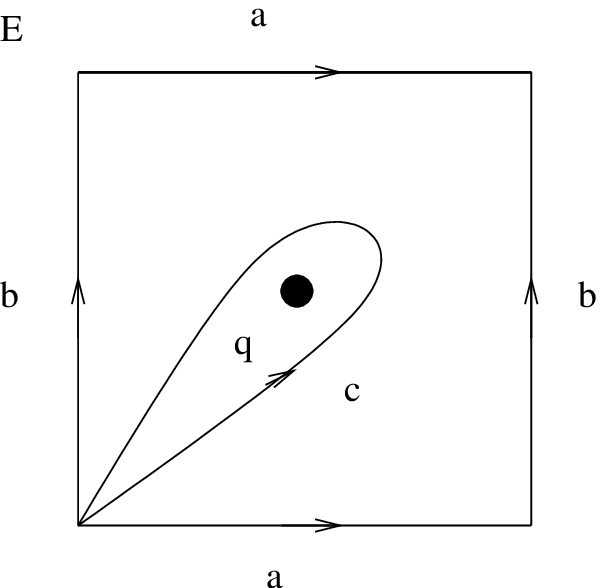}
    \end{figure}

Consider degree $d$ connected covers $\pi$ of $E$ with a unique branch point $q$ and 
$$\pi^{-1}(q) = (m_1+1)p_1 + \cdots + (m_k+1)p_k + p_{k+1} + \cdots + p_l.$$ 
Such a cover corresponds to an element in $ \mbox{Hom} (\pi_1 (E, q), S_d), $
where $S_d$ is the permutation group on $d$ letters, such that the images $\alpha, \beta, \gamma$ of $a, b, c$ satisfy 
$$\AB = \gamma \in (m_1+1)\cdots(m_k+1)(1)\cdots (1).$$ 
The notation $(m_1+1)\cdots(m_k+1)(1)\cdots (1)$ stands for the conjugacy class of $S_d$ that consists of cycles $(i)$ of length $i$. 
The cover is connected if and only if the subgroup generated by $\alpha, \beta$ acts transitively on the $d$ letters. 

Define an equivalence relation $\sim$ between two pairs $(\alpha, \beta)$ and $(\alpha', \beta')$ 
if there exists an element $\tau \in S_d$ such that $\tau (\alpha, \beta) \tau^{-1} = (\alpha',\beta')$. 
Two covers are called isomorphic if and only if there is a commutative diagram as follows: 
$$\xymatrix{
C \ar[rr]^\phi \ar[dr] & & C' \ar[dl] \\
& E  &  }$$
where $\phi$ is an isomorphism between $C$ and $C'$. 
For two isomorphic covers, their monodromy images $(\alpha, \beta)$ and $(\alpha', \beta')$ are equivalent to each other, 
since the conjugate action of $\tau$ amounts to relabeling the $d$ sheets of the cover. 

Hence, non-isomorphic degree $d$, genus $g$ connected covers of a fixed $E$ can be parameterized by the following covering set $Cov_{d,\mu}$ 
of the equivalence classes:  
$$ Cov_{d,\mu} = \{ (\alpha, \beta) \in S_d \times S_d \ |\ \AB \in (m_1+1)\cdots(m_k+1)(1)\cdots (1), \  \langle \alpha, \beta \rangle \ \mbox{is transitive}\}/\sim. $$
Varying the $j$-invariant of the elliptic curve, we obtain the 1-dimensional Hurwitz space $\T$ parameterizing such covers. 
The fiber of the finite map $e: \T\rightarrow \overline{\mathcal M}_{1,1}$ over a fixed elliptic curve
can be identified with the equivalence classes
of pairs in $Cov_{d,\mu}$. The degree $N_{d,\mu}$ of $e$ counts the number of non-isomorphic covers, hence we have 
$$ N_{d,\mu} = | Cov_{d,\mu} |. $$ 

The Teichm\"{u}ller curves $\T$ are contained in $\HH(\mu)$ as 1-dimensional SL($2,\mathbb R$) invariant orbits. If $\HH(\mu)$ has more than one connected component due to hyperelliptic, odd or even spin structures, the covering set will have a corresponding decomposition. In this case, we add subscripts $hyp$, $odd$ or $even$ to distinguish them.  

For each equivalence class of monodromy pairs $(\alpha,\beta)$ in $Cov_{d,\mu}$, suppose $\alpha$ has $a_i$ cycles of length $i$, where 
$\sum_{i=1}^{d} ia_{i} = d$. Associate to $(\alpha,\beta)$ the following weight 
$$\sum_{i=1}^{d}\frac{a_i}{i}.$$
This weight is well-defined up to equivalence, since permutations in the same conjugacy class have the same number of cycles of length $i$. 
Let $M_{d,\mu}$ denote the summation of such weights, ranging 
over the $N_{d,\mu}$ equivalence classes of monodromy pairs in $Cov_{d,\mu}$. By \cite[Theorem 1.15]{C},
the slope of $\T$ is determined by the quotient of $M_{d,\mu}$ and $N_{d,\mu}$ as follows: 
$$ s(\T) = \frac{12}{1+\kappa_{g}(\mu)N_{d,\mu}/M_{d,\mu}}, $$
where $\kappa_g(\mu) = \frac{1}{12}(2g-2+\sum_{i=1}^{k}\frac{m_i}{m_i+1})$ is a constant up to $\mu$. 
The same formula applies to $\T^{hyp}$, $\T^{odd}$ or $\T^{even}$ in case $\HH(\mu)$ has more than one connected components, and 
the summations $M, N$ are taken over the corresponding subsets of the decomposition of $Cov_{d,\mu}$. 

For the reader's convenience, we briefly explain the idea behind the slope formula. When an elliptic curve $E$ degenerates 
to a rational nodal curve, locally there is a vanishing cycle $a \in H_1 (E, \mathbb Z)$ that shrinks to the node. Let $\alpha$ be the monodromy image of $a$ 
in $S_d$ for a smooth cover in $\T$. If $\alpha$ has a cycle of length $i$, around the resulting node of the singular cover over the rational nodal curve, 
the map is locally given by $(x,y) \rightarrow (x^i, y^i)$. Moreover, one has to make a degree $i$ base change to realize a universal covering map over 
$\T\rightarrow \Mone$ around this node. From an orbifold point of view, such a node contributes $\frac{1}{i}$ to the intersection $\T\ldotp \delta$, hence
the weighted sum $M_{d,\mu}$ counts the total intersection $\T\ldotp \delta$. For $\T\ldotp \lambda$, it boils down to a standard calculation 
on the universal covering map between two families of curves over $\T$ and $\Mone$. 

To connect the slope of $\T$ with the Siegel-Veech area constant and the sum of Lyapunov exponents, we need to interpret covers of elliptic curves with a unique branch point 
as square-tiled surfaces (or origamis). Take $d$ unit squares and label them by $1,\ldots, d$. Glue their edges by a monodromy pair $(\alpha,\beta)$ and we obtain a square-tiled surface, which is the corresponding cover with respect to $(\alpha,\beta)$. See section 4 for a more detailed description and examples. The ramification points of the cover are called zeros of the square-tiled surface, which refer to zeros of the corresponding Abelian differential. Such a Riemann surface $S$ has a natural metric by pulling back $dz$. So we can talk about its geodesics with a fixed direction. A saddle connection on $S$ is a geodesic connecting two zeros. Geodesics with the same direction may fill in a maximal cylindrical area on $S$, which is bounded by saddle connections. The key observation is that the weighted sum $M_{d,\mu}$ also counts the number of maximal horizontal cylinders of height $1$ with weight $1/l$ on $S$, where $l$ is the length of a cylinder. 

\begin{proof}[Proof of Theorem~\ref{slc}]
For a smooth cover corresponding to the monodromy pair $(\alpha, \beta)$ in $Cov_{d,\mu}$, we can glue $d$ unit squares by the monodromy actions $(\alpha,\beta)$, as we will do in the examples in section 4. Let $\alpha$ correspond to the monodromy image of the horizontal edge of the torus in $S_d$. If $\alpha$ has a cycle of length $i$, say $(a_1\cdots a_i)$, 
we line up $i$ unit squares horizontally into a rectangle of length $i$ and height $1$ with two vertical edges glued together. On the resulting square-tiled surface, it corresponds to a cylinder filled in by maximal horizontal geodesics of length-$i$. 
 \begin{figure}[H]
    \centering
    \psfrag{1}{$a_1$}
    \psfrag{2}{$a_2$}
    \psfrag{j-1}{$a_{i-1}$}
    \psfrag{j}{$a_i$}
    \includegraphics[scale=0.6]{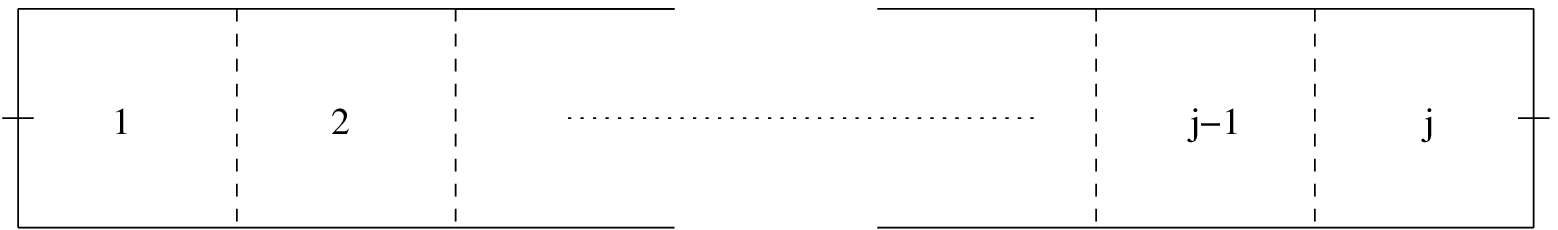}
    \end{figure}
Therefore, the weighted sum $M_{d,\mu}$ equals the total number of maximal horizontal cylinders of height $1$ with weight $1/\mbox{length}$ for all square-tiled surfaces parameterized by
$Cov_{d,\mu}$. By appendix A, the Siegel-Veech area constant $c_{\mu}$ of a stratum $\HH(\mu)$ (or one of its connected components) equals the following limit: 
$$c_{\mu} = \lim\limits_{d\to \infty} \frac{M_{d,\mu}}{N_{d,\mu}}.$$ 
Note that the summation in \cite{EKZ} for $M_{d,\mu}$ has weight height/length for a horizontal cylinder. In our setting, the square-tiled surfaces are arranged so that they only have height $1$ horizontal cylinders, hence it does make a difference if the weight is multiplied by the height of a cylinder. Combining the relation 
$$L_{\mu} = \kappa_{\mu} + c_{\mu}$$
in \cite{EKZ} with the slope formula, we have 
$$ s_{\mu} = \lim\limits_{d\rightarrow \infty}s(\T) = \frac{12c_{\mu}}{L_{\mu}}. $$
\end{proof}

The equality $c_{\mu} = \lim\limits_{d\to \infty} M_{d,\mu}/N_{d,\mu}$ used in the proof was informed to the author by Eskin. For the reader's convenience, we include a proof
in appendix A.   

\begin{remark}
We give an intuitive explanation of Theorem~\ref{slc}. By the period map, the tangent space of $\HH(\mu)$ 
can be identified with the relative cohomology group $H^1(C,p_1,\ldots,p_k,\CC)$.  
The SL$(2,\RR)$ action induces an invariant splitting 
$H^1(C,p_1,\ldots,p_k,\CC) = L_1 \oplus L_2$, 
where $L_1$ is the direction of the SL$(2,\RR)$ orbit and $L_2$ is the orthogonal complement with respect to the Hodge inner product.
Then $L_2$ supports an invariant differential form $\beta$ in $H^{2n-2}(\HH_1(\mu), \RR)$, where $\HH_1(\mu)$ is the quotient  $\HH(\mu) / \CC^{*}$ by scaling the 1-forms 
and $n = 2g-2 + k$ is the complex dimension of $\HH_1(\mu)$. Let $\gamma$ be the first Chern class of the holomorphic $ \CC^{*}$ bundle $\HH(\mu)\rightarrow \HH_1(\mu)$. 
In \cite[Section 7]{K}, there is a formula to compute the sum of Lyapunov exponents as follows: 
$$ L_{\mu} = \frac{\int_{\HH_1(\mu)}\beta\wedge\lambda}{\int_{\HH_1(\mu)}\beta\wedge\gamma}.$$ 
The Poincare dual of $\beta$ can be regarded as the limit of $\T$ as $d \rightarrow \infty$. The bottom integral is about the volume of $\HH_1(\mu)$, which can be calculated 
using the limit of $N_{d,\mu}$ \cite{EO}. Therefore, the above formula can be interpreted as 
$$ L_{\mu} = \frac{\lim\limits_{d\to \infty}\T\ldotp \lambda}{\lim\limits_{d\to \infty}N_{d,\mu}} = \frac{\lim\limits_{d\to \infty}\T\ldotp\lambda/(\frac{1}{12}\T\ldotp\delta)}{\lim\limits_{d\to \infty}N_{d,\mu}/M_{d,\mu}} = \frac{12\lim\limits_{d\to \infty}1/s(\T)}{1/c_{\mu}} = \frac{12c_{\mu}}{s_{\mu}}. $$
The coefficient 12 appears, because in the setting of the slope formula \cite[Theorem 1.15]{C}, one has to make a degree 12 base change to realize a universal covering map, e.g. take a general pencil of plane cubics. Then $12M_{d,\mu}$ counts the intersection with $\delta$ after the base change, since there are 12 rational nodal curves in the pencil. 
The reader can also refer to \cite{BM} for a related discussion of Lyapunov exponents of Teichm\"{u}ller curves by local systems. 
\end{remark}

We can calculate $c$ and $L$ explicitly for the hyperelliptic components $\HH^{hyp}(2g-2)$ and $\HH^{hyp}(g-1,g-1)$. 

\begin{proof}[Proof of Corollary~\ref{slc-hyp}]
In Corollary~\ref{slope}, for those Teichm\"{u}ller curves $\T^{hyp}$ in $\HH^{hyp}(\mu)$ for $\mu = (2g-2)$ or $(g-1,g-1)$, 
their slopes are equal to $8+\frac{4}{g}$, independent of $d$. 

For $\mu = (2g-2)$, we have
$$\kappa =  \frac{(g-1)g}{3(2g-1)}. $$
By Theorem~\ref{slc}, 
we get
$$L = \frac{12\kappa}{12-s} = \frac{g^2}{2g-1}, \  c = L - \kappa =  \frac{g(2g+1)}{3(2g-1)}. $$

For $\mu = (g-1,g-1)$, we have 
$$\kappa =\frac{(g-1)(g+1)}{6g}. $$
By Theorem~\ref{slc}, 
we get 
$$ L = \frac{12\kappa}{12-s} = \frac{g+1}{2}, \   c = L - \kappa =  \frac{(g+1)(2g+1)}{6g}. $$
\end{proof}

These numbers are also calculated in \cite{EKZ} using the determinant of the Laplacian. 

The Teichm\"{u}ller curves $\T$ were studied in \cite{C} for the purpose of bounding slopes of effective divisors on $\Mg$. For an effective divisor $D = a\lambda - b\delta $ on $\Mg$, 
define its slope 
$$s(D) = \frac{a}{b}.$$ 
If an effective curve $C$ on $\Mg$ is not contained in $D$, then $D\ldotp C\geq 0$, hence $s(D) \geq s(C)$. Fix a partition $\mu = (m_1,\ldots,m_k)$ of $2g-2$ such that $k \geq g-1$. By \cite[Theorem 1.21]{C}, the map $\HH(\mu)\rightarrow \Mg$ is dominant and the images of $\T$ cannot be all contained in $D$, since the square-tiled surfaces are integer points in $\HH(\mu)$. Hence, the limit of slopes of $\T$ as $d\to\infty$ can provide a lower bound for slopes of effective divisors on 
$\Mg$. 

There are two difficulties in this method of bounding slopes. First, the Teichm\"{u}ller curves $\T$ may be reducible. It is logically possible that a divisor $D$ contains some components of $\T$ but not all, for infinitely many $d$. In principle we also need to compute the limit of slopes of irreducible components $\Ti$ of $\T$ as $d\to\infty$. Second, we do not have an effective way to evaluate $s$, $c$ or $L$ for large $g$. Nevertheless, it is conjectured that for ``generic'' $\Ti$, the quotient $M_{d,\mu,i}/N_{d,\mu,i}$ goes to $c_{\mu}$, i.e. the limit of $s(\Ti)$ should be the same as $s$, where ``generic'' essentially means that $\Ti$ is not contained in any other closed SL$(2,\RR)$ invariant submanifold of $\HH(\mu)$. Moreover, Eskin and Zorich obtained strong numerical evidence, which predicts 
$$\lim\limits_{g\to\infty} c_{\mu}  = 2$$ for the non-hyperelliptic strata $\HH(\mu)$ (or their connected components). It is necessary to rule out hyperelliptic strata, since 
$c_{\mu}^{hyp} $ grows asymptotically as $g/3$ for $\mu = (2g-2)$ or $(g-1,g-1)$ by Corollary~\ref{slc-hyp}. 

Given the above expectations, by Theorem~\ref{slc}, an heuristic lower bound for slopes of effective divisors can be arbitrarily close to 
$$ s = \frac{12}{1+\frac{1}{2}\kappa} = \frac{288}{2g+22+\displaystyle\sum_{i=1}^k \frac{m_i}{m_i+1}} $$
for large $g$, where $(m_1,\ldots, m_k)$ is a partition of $2g-2$ and $k \geq g-1$. 
The partition $(1,\ldots, 1, g)$ maximizes the above bound and we have $s \sim 576/5g$ as $g\to \infty$ for this partition. An interesting feather is that this asymptotic bound
$576/5g$ also appeared in \cite{HM1} based on an heuristic analysis of monodromy data, where Harris and Morrison studied 1-dimensioanl families $Z_{k,g}$ of degree $k$, genus $g$ simply branched covers of a rational curve by varying a branch point. Another lower bound $60/(g+4)$ was established by Pandharipande \cite{P} using Hodge integrals. To the author's best knowledge, on the one hand, there is no known effective divisor on $\Mg$ with slope $\leq 6$. On the other hand, there is no known lower bound better than $O(1/g)$ for slopes of effective divisors as $g$ goes to infinity. The range $[O(1/g), 6]$ has been mysterious for a long time. The reader can refer to \cite{Mo} for a nice survey on this problem and its recent development. 
In any case, we come up with the following conjecture, which predicts that $O(1/g)$ is more likely to be the sharp lower bound for slopes of effective divisors. 

\begin{conjecture}
There exist effective divisors on $\Mg$ with slopes arbitrarily close to $576/5g$ as $g$ goes to infinity. 
\end{conjecture}

The author has not been able to find further evidence for this conjecture except the following reasoning. 

\noindent {\bf An heuristic argument.} Assume that $576/5g$ is the limit of slopes of $\T$ for $\mu = (1,\ldots, 1, g)$ as $d,g\to \infty$, and that the slopes of moving curves $Z_{k,g}$ in \cite{HM1} grow like $576/5g$ as expected for $g\gg 0$ and $k = [(g+3)/2]$. Let $R$ be a curve class in $N_1 (\Mg) \otimes \QQ$ such that $R\ldotp \delta_i = 0$ for $1\leq i \leq [g/2]$ 
and $s(R) = (R\ldotp \delta_0)/(R\ldotp\lambda) = 576/5g$. Since $\T$ does not meet $\delta_i$ for any $i > 0$ \cite[Proposition 3.1]{C} and $Z_{d,g}\ldotp \delta$ is dominated by $Z_{k,g}\ldotp \delta_0$ for large $g$, the rays spanned by the limit of $\T$ as $d\to\infty$ and by $Z_{k,g}$ should be close to $R$ as $g\to \infty$. Moreover, $Z_{k,g}$ is moving on $\Mg$ but $\T$ is rigid. It is natural to expect that their limits are close to an extremal ray $R'$ of the cone of moving curves on $\Mg$ for $g\gg 0$, which is dual to the pseudo-effective cone of divisors by \cite{BDPP}. Given such an expectation, the dual face of $R'$ in the pseudo-effective cone would be spanned by $\delta_1, \ldots, \delta_{[g/2]}$ along with a pseudo-effective divisor $D$ such that $D\ldotp R' = 0$. Since $R'$ is close to $R$, the slope of $D$ would be close to $s(R) = 576/5g.$ 

\begin{remark}
An extremal ray of the cone of moving curves on a projective variety may be generated by both a moving curve and a rigid curve. See Remark~\ref{movingrigid} for an example.  
Such examples also include rational curves $R$ on a quintic threefold $X$ with normal bundle $N \cong \OO(-1)\oplus \OO(-1)$. Since $h^{0}(f^{*}N) = 0$ for any finite morphism $f$ to $R$, 
the curve $R$ is called super rigid in the sense that any multiple of $R$ does not deform in $X$. Note that the Picard number of $X$ equals one, so $R$ is numerically equivalent to a moving curve class. \end{remark}

By Theorem~\ref{slc} and the Siegel-Veech constants $c_{\mu}$ calculated in \cite{EKZ}, in appendix B we list the limit $s_{\mu}$ of slopes of $\T$ in each stratum $\HH(\mu)$ for small genus. If an effective divisor $D$ on $\Mg$ has slope smaller than $s_{\mu}$, it must contain the images of infinitely many $\Ti$. This can induce various applications for understanding the geometry of $D$. For instance, the Brill-Noether divisor for $g=5$ has slope 8 and $s^{even}_{(8)} =  152064/18959 > 8$. Then this divisor contains infinitely many 
genus 5 elliptic coverings that admit a canonical divisor of form $(8p)$, where $p$ is the total ramification point and the Theta divisor $4p$ yields an even spin structure. 

There is a similar slope problem on $\mathcal A^{*}_g$, which is a partial compactification of the moduli space of $g$-dimensional principally polarized abelian 
varieties. The lower bound for slopes of effective divisors on $\mathcal A^{*}_g$ is known to approach zero as $g$ goes to infinity. In fact, there exists an effective divisor 
on $\mathcal A^{*}_g$ of slope at most 
$$\frac{(2\pi)^2}{\big(2(g!)\zeta(2g)\big)^{1/g}}.$$ 
Since $\lim \limits_{g\to \infty} \zeta(2g) = 1$ and $(g!)^{1/g} \sim g/e$, 
it is easy to check that the slope of this divisor is smaller than $576/5g$ for $g \gg 0$. See \cite{Gr} for a good introduction on this topic and further references. 

\begin{corollary}
Given $\lim\limits_{g\to\infty} c_{\mu}  = 2$ for $\mu = (1,\ldots, 1, g)$, a divisor on $\mathcal A^{*}_g$ defined by a modular form of slope smaller than $576/5g$ must contain the image of $\mathcal M_g$ via the Torelli embedding $\mathcal M_g \hookrightarrow \mathcal A_g$ for $g\gg 0$. 
\end{corollary}

\begin{proof}
We have seen that slopes of effective divisors on $\Mg$ are bounded by $576/5g$ from below for $g\gg 0$, 
assuming that $\lim\limits_{g\to\infty} c_{\mu}  = 2$ for $\mu = (1,\ldots, 1, g)$. If an effective divisor $D$ on $\mathcal A^{*}_g$ does not contain $\mathcal M_g$, its restriction 
to $\mathcal M_g$ induces an effective divisor on $\Mg$ with the same slope. Hence, the slope of $D$ cannot be smaller than $576/5g$ for $g\gg 0$. 
\end{proof}

\section{Square-tiled surfaces with a unique zero}
A cover of an elliptic curve with a unique branch point can be realized as a lattice polygon whose edges are glued with respect to 
the monodromy pair $(\alpha, \beta)$. Such a cover is called a square-tiled surface, which has been studied 
intensively from the viewpoint of dynamics, cf. \cite{HS} for an introduction and related references. There is a correspondence between 
square-tiled surfaces and the above monodromy pairs. Let $\pi: C\rightarrow E$ be a degree $d$ cover 
over the standard torus $E$ with a unique branch point. Take $d$ unit squares and mark them by $1, \ldots, d$. 
For the $i$-th unit square, mark its upper and lower horizontal edges by $b_i$ and $b_i'$, respectively. Similarly, mark its right and left 
vertical edges by $a_i$ and $a_i'$, respectively. Let $(\alpha, \beta)$ acting on the $d$ letters $\{1, \ldots, d \}$ 
be the monodromy pair corresponding to a degree $d$ cover $\pi: C\rightarrow E$ uniquely branched at 
the vertices of $E$. One can realize $C$ as a square-tiled surface of area $d$ by 
identifying $a_i$ with $a_{\alpha(i)}'$ and identifying $b_i$ with $b_{\beta(i)}'$ by parallel transport.
If $\pi$ has a unique ramification point, some vertices of the squares will be glued to form the unique zero of 
the square-tiled surface, which corresponds to the zero of the Abelian differential by pulling back $dz$. 

Consider the partition $\mu = (2g-2)$. The covering set is reduced to 
$$ Cov_{d,(2g-2)} = \{ (\alpha, \beta) \in S_d \times S_d \ |\ \AB \in (2g-1)(1)\cdots (1), \  \langle \alpha, \beta \rangle \ \mbox{is transitive}\}/\sim. $$
By the description in section 3, the fiber of the finite map $e: \mathcal T_{d,(2g-2)}\rightarrow \mathcal M_{1,1}$ can be identified with the equivalence classes
of pairs in $Cov_{d,(2g-2)}$. The degree $N_{d,(2g-2)}$ of $e$ counts the number of non-isomorphic covers, hence we have 
$$ N_{d,(2g-2)} = | Cov_{d,(2g-2)}  |. $$ 

\begin{example}
\label{five}
Let $\alpha = (1234)(5)$ and $\beta = (15)(2)(3)(4)$ be two permutations on five letters. We have 
$\AB = (154)(2)(3)$, so this monodromy pair yields a degree 5, genus 2 cover of a torus with a unique ramification point. The corresponding square-tiled 
surface looks like the following: 
\begin{figure}[H]
    \centering
    \psfrag{1}{$1$}
    \psfrag{2}{$2$}
    \psfrag{3}{$3$}
    \psfrag{4}{$4$}
    \psfrag{5}{$5$}
    \includegraphics[scale=0.6]{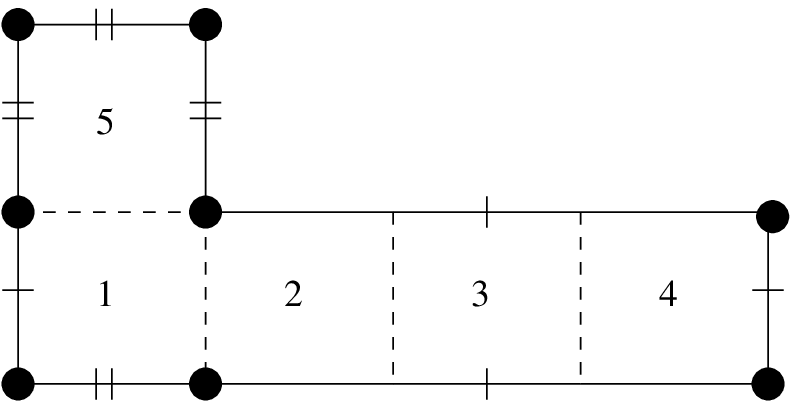}
    \end{figure}
It is an octagon whose edges with the same label are glued in pairs. The eight vertices are glued together as the unique zero. 
It is clear that along the horizontal direction, the permutation is $(1234)(5)$ and along the vertical direction, the permutation is 
 $(15)(2)(3)(4).$
\end{example}

We have seen that there are infinitely many genus $g$ covers of a torus with a unique ramification point in the hyperelliptic locus. 
Veech \cite{V1} showed that a hyperelliptic curve can be obtained by gluing the opposite sides of a centrally symmetric $2n$-gon. 
We want to study which monodromy pairs in $Cov_{d,g}$ admit square-tiled surfaces that are hyperelliptic. 
Note that a genus $g$ curve $C$ is hyperelliptic if and only if $C$ has an involution with $2g+2$ fixed point. 
One would naturally consider covers with order two automorphisms, but the following result suggests this is not possible. 

\begin{proposition}
For a cover $\pi: C\rightarrow E$ parameterized by $\mathcal T_{d,(2g-2)}$, it does not have non-trivial automorphisms of order two. 
\end{proposition}

\begin{proof}
Let $(\alpha, \beta)$ be a monodromy pair in $Cov_{d,(2g-2)}$ associated to $\pi$. An order two automorphism of $\pi$ corresponds to a simple transposition $\tau$ 
in $S_d$ such that $\tau (\alpha, \beta) \tau^{-1} = (\alpha, \beta).$ Let $\{1,\ldots, d\}$ be the $d$ letters and assume that 
$\AB = \gamma = (12\cdots (2g-1))(2g)\cdots (d).$ We get $\tau\gamma\tau^{-1} = \gamma$, hence $\tau (12\cdots (2g-1)) \tau^{-1} = (12\cdots (2g-1)).$ 
Then $\tau$ has to send $i$ to $i+m$ (mod $2g-1$) for a fixed $m$ and any $1\leq i \leq 2g-1$. Since $\tau^{2} = id$, we have $i+2m = i$ (mod $2g-1$), which is impossible. 
\end{proof}

Nevertheless, the elliptic curve $E$ has an involution $\iota$. One may hope to find covers that have an involution $\imath$ compatible 
with the elliptic involution:  
$$\xymatrix{
C \ar[r]^{\imath} \ar[d]_{\pi}  &  C\ar[d]^{\pi}  \\
 E\ar[r]^{\iota}           & E}
$$ 
This can be characterized by the monodromy pair $(\alpha, \beta)$ associated to $\pi$ as follows. 

\begin{proposition}
\label{involution}
A cover $\pi: C\rightarrow E$ admits an involution $\imath$ as in the above diagram if and only if 
there exists a simple transposition $\tau\in S_d$ such that $\tau (\alpha, \beta) \tau^{-1} = (\alpha^{-1}, \beta^{-1})$.  
For such $\tau$, let $n$, $n_a$ and $n_b$ denote the numbers of fixed letters in $\{1,\ldots, d  \}$ by $\tau$, $\tau\alpha$ and $\tau\beta$, 
respectively. Let $n_{ba}$ denote the number of letters fixed by both $\tau\beta\alpha$ and $\AB$. 
Then the involution $\imath$ has $n + n_a + n_b + n_{ba} + 1$ fixed points. In particular, $C$ is hyperelliptic 
if $n + n_a + n_b + n_{ba} = 2g+1$. 
\end{proposition}

\begin{proof}
The composite map $\pi \circ \iota: C\rightarrow E $ has a unique ramification point, so it corresponds to a cover $\pi'$
parameterized by $\mathcal T_{d,(2g-2)}$. Since $\iota$ sends the basis $(a,b)$ of the fundamental group $\pi_1(E, q)$ to $(a^{-1}, b^{-1})$, 
the monodromy pair associated to $\pi'$ is given by $(\alpha^{-1}, \beta^{-1})$. Note that $\pi$ and $\pi'$ are isomorphic if and only if 
there exists $\tau \in S_d$ such that $\tau (\alpha, \beta) \tau^{-1} = (\alpha^{-1}, \beta^{-1})$. Moreover, $\tau^{2} = id$ so $\tau$ is
a simple transposition. 

Take the standard torus $E$ by gluing the parallel edges of $[0,1]^2$. We can realize $C$ as a square-tiled surface by gluing $d$ unit squares with respect to the monodromy pair $(\alpha, \beta)$.  Note that $\iota$ has four $2$-torsion points, the center of the square, the center of each edge and the vertices. The fixed points of $\imath$ on the square-tiled surface $C$ can only occur at these 2-torsion points. 
Mark them on the $i$-th unit square as follows: 
\begin{figure}[H]
    \centering
    \psfrag{Ai}{$A_i$}
    \psfrag{Bi}{$B_i$}
    \psfrag{Ci}{$C_i$}
    \psfrag{Xi}{$X_i$}
    \psfrag{Yi}{$Y_i$}
     \psfrag{Zi}{$Z_i$}
    \psfrag{Wi}{$W_i$}
    \psfrag{Ai'}{$A_i'$}
    \psfrag{Bi'}{$B_i'$}
    \includegraphics[scale=0.5]{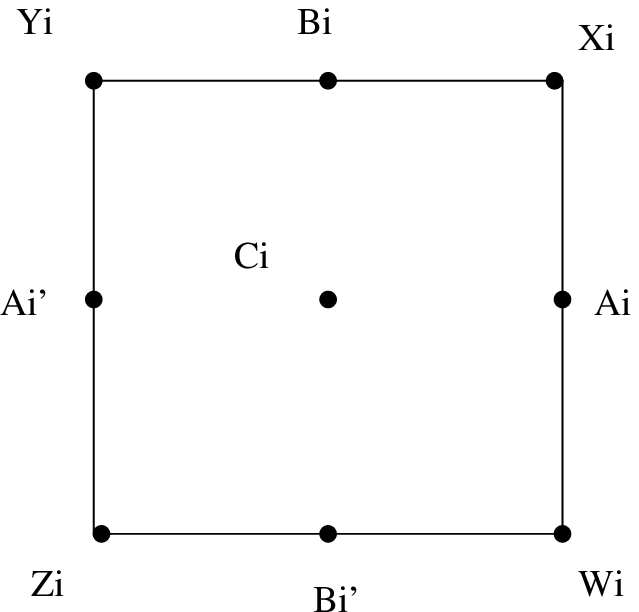}
    \end{figure}

The action $\imath$ sends the $i$-th square to the $\tau(i)$-th by parallel transport and then sends each point to 
its conjugate point symmetric to the center of the square. We have 
$\imath (C_i) = C_{\tau_i}, \imath (A_i) = A_{\tau(i)}', \imath (B_i) = B_{\tau(i)}', \imath (X_i) = Z_{\tau(i)}$ and $\imath(Y_i) = W_{\tau(i)}. $ 
Then $C_i$ is fixed by $\imath$ if and only if $i$ is fixed by $\tau$. 
Recall that $C$ is obtained by identifying the edges $a_i$ with $a_{\alpha(i)}'$ and identifying $b_i$ with $b_{\beta(i)}'$ by parallel transport.
In this process, $A_i$ is glued to $A_{\alpha(i)}'$ and $B_i$ is glued to $B_{\beta(i)}'.$ 
Hence, $A_i$ is fixed by $\imath$ if and only if $\tau(i) = \alpha(i)$, namely, 
$i$ is fixed by $\tau\alpha$ since $\tau$ is a simple transposition. Similarly,  $B_i$ is fixed by $\imath$ if and only if $i$ is fixed by $\tau\beta$. 
The vertices $X_i, Y_{\alpha(i)}, Z_{\beta\alpha(i)}, W_{\alpha^{-1}\beta\alpha(i)}, X_{\AB(i)}, \ldots $ are identified as the same point 
on $C$. If $\AB(i) = i$, we get an unramified integer point over the unique branch point. Since $\imath (X_i) = Z_{\tau(i)}$ and $\imath(Y_{\alpha(i)}) = W_{\tau\alpha(i)}$, 
this point is fixed by $\imath$ if and only if $\tau(i) = \beta\alpha(i)$ or $\alpha^{-1}\beta\alpha(i) = \tau\alpha(i)$. 
Namely, $i$ is fixed by $\tau\beta\alpha$. 
If $\AB(i) \neq i$, we get the unique ramification point on $C$, which is obviously fixed by $\imath$. 

In sum, $\imath$ has $n$ fixed points at the center of the $d$ squares, $n_a$ fixed points at the center of the vertical edges, $n_b$ 
fixed points at the center of horizontal edges and $n_{ba} + 1$ fixed integer points. The total number of fixed points is equal to 
$n + n_a + n_b + n_{ba} + 1$. If this equals $2g+2$, by Riemann-Roch, $C\rightarrow C/\imath$ is a double cover of $\Po$, hence $C$ is hyperelliptic.
\end{proof}

\begin{example}
In Example~\ref{five}, we have the monodromy pair $\alpha = (1234)(5)$ and $\beta = (15)(2)(3)(4)$. 
Take $\tau = (24)(1)(3)(5)$ and one can check that $\tau (\alpha, \beta)\tau^{-1} = (\alpha, \beta)$. 
By Proposition~\ref{involution}, $C$ admits an involution $\imath$ compatible with the elliptic involution. 
The permutation $\tau = (24)(1)(3)(5)$ fixes the three letters $1, 3, 5$, $\tau\alpha = (14)(23)(5)$ fixes
the letter $5$ and $\tau\beta = (15)(24)(3)$ fixes the letter $3$. Note that $\tau\beta\alpha = (145)(23)$ has no fixed letter. 
Overall, we get $n = 3, n_a = 1, n_b = 1, n_{ba} = 0$, hence $\imath$ has six fixed points. This coincides with the fact that $C$ is a genus two hyperelliptic curve
with six Weierstrass points. 

Using the square-tiled surface model 
for $C$, $\imath$ and its fixed points can be seen as follows: 
 \begin{figure}[H]
    \centering
    \psfrag{1}{$1$}
    \psfrag{2}{$2$}
    \psfrag{3}{$3$}
    \psfrag{4}{$4$}
    \psfrag{5}{$5$}
    \includegraphics[scale=0.7]{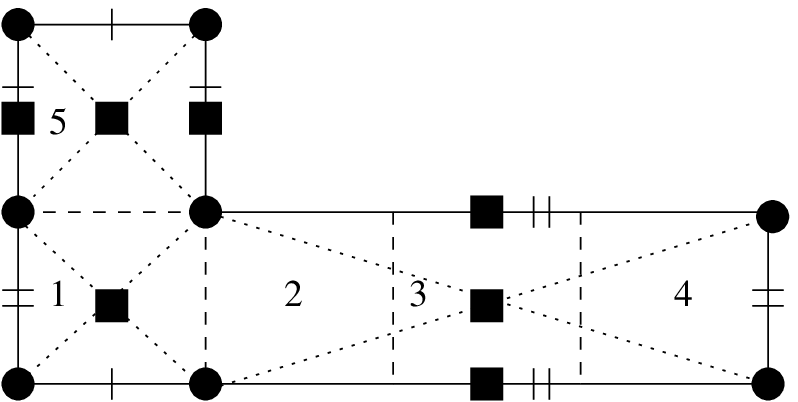}
    \end{figure}
Among the six Weierstrass points, three are at the center of the square, one is on the horizonal edge, one is on the vertical edge, and the last one 
is the unique zero of the 1-form by pulling back $dz$. 
\end{example}

\begin{example}
\label{53}
Consider $d=5$, $g=3$ and the partition $\mu = (4)$. 
Let $\{1,2,3,4,5\}$ be the letters labeled on the five sheets of a cover parameterized by $\mathcal T_{5,(4)}$.  
Consider all possible $\alpha, \beta \in S_5$ such that $\AB$ consists of a single cycle of length 5. 
A routine examination shows that $Cov_{5,(4)}$ parameterizes 40 equivalence classes represented by the following list: \\
(1) $\alpha = (12)(34)(5), \beta = (12345);$ (2) $\alpha = (12)(35)(4),
\beta = (12345);$ \\
(3) $\alpha = (124)(3)(5), \beta = (12345);$ (4) $\alpha = (142)(3)(5), \beta
= (12345);$ \\
(5) $\alpha = (12453), \beta = (12345);$ (6) $\alpha = (13254),
\beta = (12345);$ \\
(7) $\alpha = (14)(25)(3),\beta=(123)(4)(5);$ (8) $\alpha = (12435), \beta
= (123)(4)(5); $ \\
(9) $\alpha = (13425), \beta = (123)(4)(5);$ (10) $\alpha = (15)(23)(4),
\beta = (12)(34)(5);$ \\
(11) $\alpha = (135)(2)(4),\beta = (12)(34)(5);$ (12) $\alpha = (12345),
\beta = (12)(34)(5);$\\
(13) $\alpha = (12354), \beta = (12)(34)(5);$ (14) $\alpha = (1243)(5),
\beta = (12345);$ \\
(15) $\alpha = (1342)(5), \beta = (12345);$ (16) $\alpha = (15)(23)(4),
\beta = (1234)(5);$ \\
(17) $\alpha = (15)(24)(3), \beta = (1234)(5);$ (18) $\alpha = (15)(34)(2),
\beta = (1234)(5);$ \\
(19) $\alpha = (135)(2)(4),\beta = (1234)(5);$ (20) $\alpha = (125)(34),
\beta = (1234)(5);$ \\
(21) $\alpha = (152)(34),\beta = (1234)(5);$ (22) $\alpha = (1325)(4),
\beta = (1234)(5);$\\
(23) $\alpha = (1352)(4),\beta = (1234)(5);$ (24) $\alpha = (1523)(4),\beta
=(1234)(5);$ \\
(25) $\alpha = (1253)(4),\beta = (1234)(5);$ (26) $\alpha = (12435),
\beta = (1234)(5);$ \\
(27) $\alpha = (14235), \beta = (1234)(5);$ (28) $\alpha = (14)(23)(5),
\beta = (123)(45);$ \\
(29) $\alpha = (124)(3)(5), \beta = (123)(45);$ (30) $\alpha = (134)(2)(5),
\beta = (123)(45);$ \\
(31) $\alpha = (145)(23),\beta = (123)(45);$ (32) $\alpha =
(1245)(3), \beta = (123)(45);$ \\
(33) $\alpha = (1345)(2), \beta = (123)(45);$ (34) $\alpha = (1425)(3),
\beta = (123)(4)(5);$ \\
(35) $\alpha = (124)(35), \beta = (123)(4)(5);$ (36) $\alpha =
(142)(35), \beta = (123)(4)(5);$ \\
(37) $\alpha = (143)(25), \beta = (12)(34)(5);$ (38) $\alpha =
(1345)(2), \beta = (12)(34)(5);$ \\
(39) $\alpha = (1354)(2), \beta = (12)(34)(5);$ (40) $\alpha = (1534)(2),
\beta = (12)(34)(5).$ 

By \cite[Theorem 1.18]{C}, there is a group of actions generated by $h_{\alpha}: (\alpha, \beta) \rightarrow (\alpha, \alpha\beta)$ 
and $h_{\beta}: (\alpha, \beta)\rightarrow (\beta\alpha, \beta)$ on $Cov_{d,g}$. Each orbit of the actions corresponds to an irreducible 
component of $\mathcal T_{d,(2g-2)}$. The above 40 pairs fall into 4 orbits, hence $\mathcal T_{5,(4)}$ has 4 irreducible components: \\
(2),(10),(13) belong to the first component $\mathcal T_{5,(4),1}$; \\
(1),(3),(4),(5),(6),(7),(8),(9),(11),(12) belong to the second
component $\mathcal T_{5,(4),2}$; \\
(14),(15),(16),(18),(22),(23),(24),(25),(26),(27),(38),(40) belong
to the third component $\mathcal T_{5,(4),3}$; \\
(17),(19),(20),(21),(28),(29),(30),(31),(32),(33),(34),(35),(36),(37),(39) belong to the last component
$\mathcal T_{5,(4),4}$. 

Using the slope formula \cite[Theorem 1.15]{C}, we obtain the slope of each component as follows: 
$$ s(\mathcal T_{5,(4),1}) = s(\mathcal T_{5,(4),4}) = 9\frac{1}{3} $$
and
$$ s(\mathcal T_{5,(4),2}) = s(\mathcal T_{5,(4),3}) = 9.$$ 

Since the hyperelliptic divisor $\overline{H}_3$ on $\overline{\mathcal M}_3$ has slope $9$, it has negative intersection with 
$\mathcal T_{5,(4),1}, \mathcal T_{5,(4),4}$, and zero intersection with $\mathcal T_{5,(4),2}, \mathcal T_{5,(4),3}.$ It implies that covers parameterized by
$\mathcal T_{5,(4),1}$ and $\mathcal T_{5,(4),4}$ are hyperelliptic. They provide two irreducible rigid curves on $\widetilde{M}_{0,8}$. 
Note that they have slope $9\frac{1}{3}$ as predicted in Corollary~\ref{slope}. 
The other two components $\mathcal T_{5,(4),2}$ and $\mathcal T_{5,(4),3}$ do not intersect $\overline{H}_3$. They map into the divisor on 
$\overline{\mathcal M}_3$ whose general points parameterize plane quartics with a hyperflex line. 

We can also analyze the components from the viewpoint of square-tiled surfaces. Take a cover $\pi: C\rightarrow E$ corresponding to 
the case (13) $\alpha = (12354), \beta = (12)(34)(5)$
in $\mathcal T_{5,(4),1}.$ We can choose a simple transposition $\tau = (12)(34)(5)$ such that $\tau (\alpha, \beta) \tau^{-1} = (\alpha^{-1}, \beta^{-1}). $
By Proposition~\ref{involution}, we have $n= 1, n_a = 1, n_b = 5, n_{ba} = 0$ and $n + n_a + n_b + n_{ba} = 7$, so $C$ is hyperelliptic. 
The hyperelliptic involution of $C$ and its eight Weierstrass points can be seen as follows: 
 \begin{figure}[H]
    \centering
    \psfrag{a}{$a$}
    \psfrag{b}{$b$}
    \psfrag{c}{$c$}
    \psfrag{d}{$d$}
    \psfrag{e}{$e$}
    \psfrag{f}{$f$}
     \psfrag{1}{$1$}
    \psfrag{2}{$2$}
    \psfrag{3}{$3$}
    \psfrag{4}{$4$}
    \psfrag{5}{$5$}
    \includegraphics[scale=0.6]{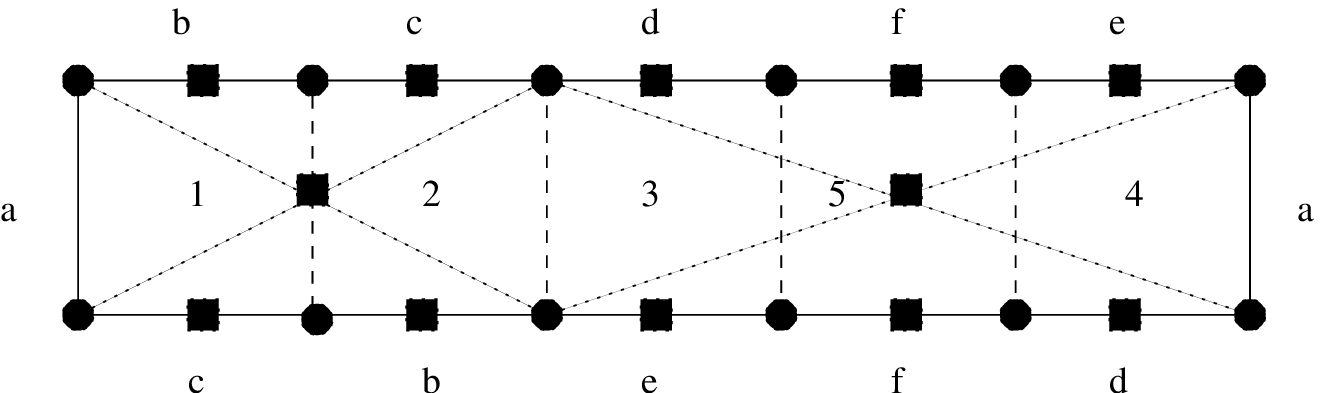}
    \end{figure}
    
Similarly, for the case (39) $\alpha = (1354)(2), \beta = (12)(34)(5)$ in $\mathcal T_{5,(4),4}$, the square-tiled surface has a hyperelliptic involution 
and its eight Weierstrass points are as follows: 
 \begin{figure}[H]
    \centering
    \psfrag{a}{$a$}
    \psfrag{b}{$b$}
    \psfrag{c}{$c$}
    \psfrag{d}{$d$}
    \psfrag{e}{$e$}
    \psfrag{f}{$f$}
     \psfrag{1}{$1$}
    \psfrag{2}{$2$}
    \psfrag{3}{$3$}
    \psfrag{4}{$4$}
    \psfrag{5}{$5$}
    \includegraphics[scale=0.7]{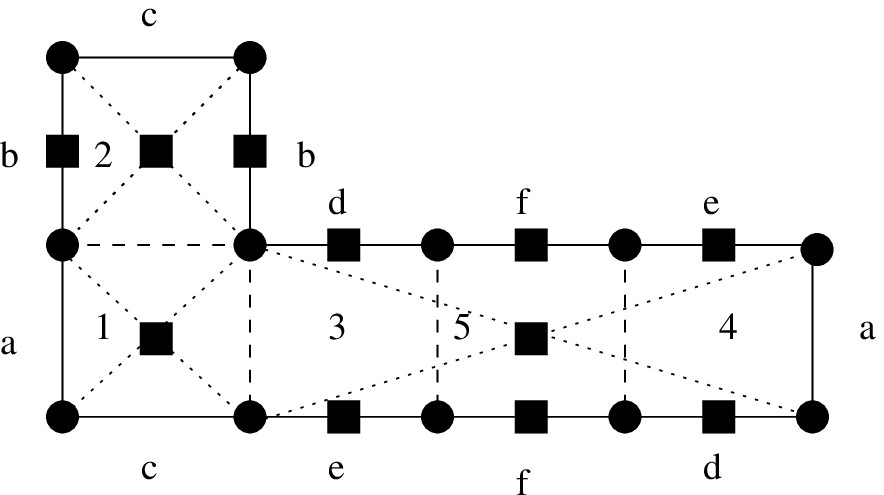}
    \end{figure}

Consider the case (8) $\alpha = (12435), \beta= (123)(4)(5)$ in $\mathcal T_{5,(4),2}$. If a simple transposition $\tau$ 
satisfies $\tau (\alpha,\beta) \tau^{-1} = (\alpha,\beta)$, one can check that $\tau$ has to be $(12)(45)(3)$.  
Note that $n = 1, n_a =  1, n_b = 1, n_{ba} = 0$. By Proposition~\ref{involution}, the involution induced by $\tau$ is not hyperelliptic,
but a double cover of an elliptic curve. So a covering curve in $\mathcal T_{5,(4),2}$ is not only a degree five cover of an elliptic curve with a unique ramification point, 
but also a double cover of another elliptic curve. See the following square-tiled surface: 
\begin{figure}[H]
    \centering
    \psfrag{a}{$a$}
    \psfrag{b}{$b$}
    \psfrag{c}{$c$}
    \psfrag{d}{$d$}
    \psfrag{e}{$e$}
    \psfrag{f}{$f$}
     \psfrag{1}{$1$}
    \psfrag{2}{$2$}
    \psfrag{3}{$3$}
    \psfrag{4}{$4$}
    \psfrag{5}{$5$}
    \includegraphics[scale=0.6]{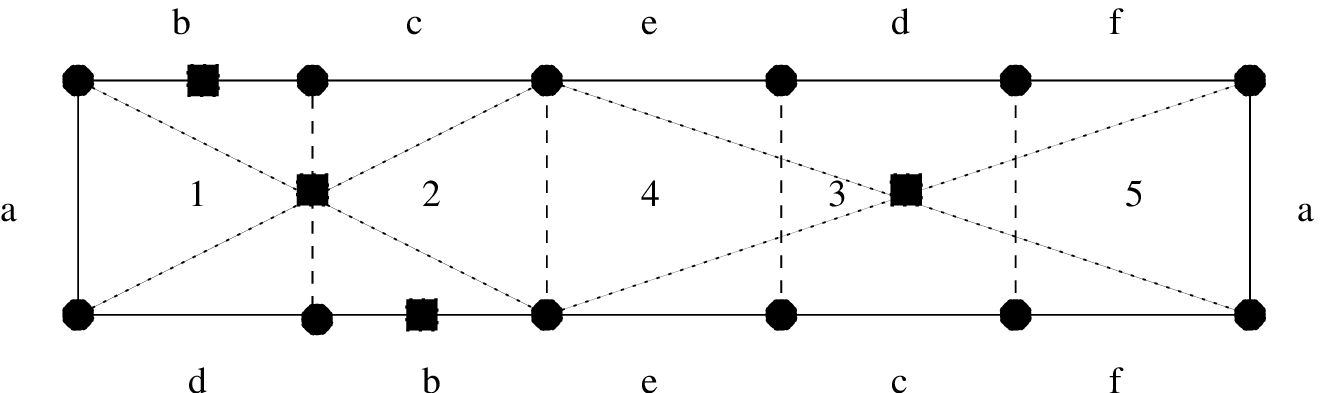}
    \end{figure}

For $(\alpha, \beta)$ in $\mathcal T_{5,(4),3}$, one can check that there does not exist such $\tau$ as in Proposition~\ref{involution}. So 
those covers do not admit an involution compatible with the elliptic involution. Correspondingly in their square-tiled surface models, there is no 
symmetry of order two. 
\end{example}

\section{Appendix A: Siegel-Veech area constants}
In this paper, we denote by $c_{\mu} $ the Siegel-Veech area constant of the stratum $\HH(\mu)$, which equals $\frac{\pi^2}{3}c_{area}$ in the context of \cite{EKZ}.  
In the proof of Theorem~\ref{slc}, we use an equality
\begin{equation}
\label{1}
c_{\mu} = \lim\limits_{d\to \infty} \frac{M_{d,\mu}}{N_{d,\mu}}. 
\end{equation}
Here we include a proof of (1), which was explained to the author by Eskin. The reader can refer to \cite{EKZ} and \cite{EMZ} for broader discussions on Siegel-Veech constants. 

We first fix some notation. Let $S_0 = (C, \omega)$ be a square-tiled surface that comes from a genus $g$ standard torus covering $\pi: C\rightarrow E$ with a unique branch point and the ramification class $\mu$, where $\omega = \pi^{-1}(dz)$ and $\mu$ is a partition of $2g-2$. As a point in the stratum $\HH(\mu)$ of Abelian differentials, $S_0$ has integer coordinates under the period map. The SL$(2,\RR)$ action on the real and imaginary parts of $\omega$ induces an action on $\HH(\mu)$. Let $\HH_1(\mu)$ be the subset of area one surfaces, which is 
SL$(2,\RR)$ invariant in $\HH(\mu)$. Let $\Gamma(S_0)$ be the group of stabilizers of the SL$(2,\RR)$ action on $S_0$, which is called the Veech group. Use $\OO(S_0)$ to denote the SL$(2,\ZZ)$ orbit of $S_0$. For simplicity, let $G =$ SL$(2,\RR)$, $\Gamma = $ SL$(2,\ZZ)$ and $\eta$ be the normalized Haar measure on $G$ such that $\eta(G/\Gamma) = 1$. 

Let $\phi: \HH_1(\mu) \rightarrow \RR$ be any $L^1$ function. By expressing the fundamental domain of $G/\Gamma(S_0)$ as a union 
of the fundamental domains of $G/\Gamma$, we have the following standard equality  
\begin{equation}
\label{2}
\int_{G/\Gamma(S_0)} \phi(gS_0)d\eta(g) = \sum\limits_{S\in \OO(S_0)} \int_{G/\Gamma}\phi(gS)d\eta(g). 
\end{equation}

Since $\omega$ induces a flat structure on $S$, we can talk about geodesics and cylinders on $S$ with a fixed direction. 
Let $f: \RR^2 \rightarrow \RR$ be a bounded function of compact support and define $\hat{f}: \HH_1(\mu)\rightarrow \RR$  
by $$ \hat{f}(S) = \sum\limits_{C\in Cyl(S)} f(\vec{C}) Area(C), $$
where $Cyl(S)$ is the set of cylinders on the flat surface $S$, $\vec{C}\in \RR^2$ is the associated vector of $C\in Cyl(S)$ and 
$Area(C)$ is the area of the cylinder. 

Let $\nu$ be the Lebesgue measure on $\HH_1(\mu)$, cf. \cite[Definition 1.3]{EO}. We know that $\HH_1(\mu)$ has finite volume under $\nu$. By the Siegel-Veech formula \cite{V2}, the Siegel-Veech constants $c_{area}(S_0)$ 
and $c_{\mu}$ satisfy the following property that for the above $f$, 
\begin{equation}
\label{3}
\frac{1}{\eta(G/\Gamma(S_0))}\int_{G/\Gamma(S_0)} \hat{f}(gS)d\eta(g) = c_{area}(S_0)\int_{\RR^2} f, 
\end{equation}
and
\begin{equation}
\label{4}
\frac{1}{\nu(\HH_1(\mu))}\int_{\HH_1(\mu)} \hat{f}(S)d\nu(S) = c_{\mu}\int_{\RR^2} f.
\end{equation}

By \cite{EKZ}, we know 
\begin{equation}
\label{5}
c_{area}(S_0) = \frac{1}{|\OO(S_0)|}\sum\limits_{S\in \OO(S_0)} w(S), 
\end{equation}
where $w(S)$ is the sum of height($C$)/length($C$) over all horizontal cylinders $C$ on $S$. 

Recall the covering set of equivalence classes $Cov_{d,\mu}$ introduced in section 3. It parameterizes degree $d$ connected covers of a fixed torus with a unique branch point of ramification type $\mu$. We also defined two summations 
$$N_{d,\mu} = |Cov_{d,\mu}|$$ 
and 
$$M_{d,\mu} = \sum\limits_{S\in Cov_{d,\mu}} w(S). $$

The group $\Gamma = $ SL$(2, \ZZ)$ acts on $Cov_{d,\mu}$, which is the same as the monodromy action \cite[Theorem 1.18]{C}. Each of its orbits corresponds to an irreducible component $\Ti$ of the Hurwitz space $\T$. Let $\Delta$ be a subset of $Cov_{d,\mu}$ consisting of a single element from each orbit. Note that in our setting, we have $\eta(G/\Gamma(S_0)) = |\OO(S_0)|$. 
Now by (2), (3) and (5), we have 
\begin{align}
\label{6}
\sum\limits_{S\in Cov_{d,\mu}}\int_{G/\Gamma} \hat{f}(gS)d\eta(g) & =  \sum\limits_{S_0\in \Delta}\sum\limits_{S\in\OO(S_0)}\int_{G/\Gamma}\hat{f}(gS)d\eta(g) \\
& =  \sum\limits_{S_0\in\Delta}\int_{G/\Gamma(S_0)} \hat{f}(gS)d\eta(g) \nonumber\\
& =  \sum\limits_{S_0\in \Delta} |\OO(S_0)|c_{area}(S_0) \int_{\RR^2}f \nonumber\\
& =  \sum\limits_{S_0\in \Delta}\sum\limits_{S\in\OO(S_0)} w(S)\int_{\RR^2}f \nonumber\\
& =  M_{d,\mu} \int_{\RR^2}f.\nonumber
\end{align}

By the argument in \cite{EO}, for such a function $\hat{f}$ on $\HH_1(\mu)$ and any given $g\in G$, we have 
\begin{equation} 
\label{7}
\lim\limits_{d\to \infty}\frac{1}{N_{d,\mu}} \sum\limits_{S\in Cov_{d,\mu}} \hat{f}(gS) = \frac{1}{\nu(\HH_1(\mu))}\int_{\HH_1(\mu)}\hat{f}(S)d\nu(S).
\end{equation}

By (6), (7) and $\eta(G/\Gamma) = 1$, we have 
\begin{align}
\label{8}
\lim\limits_{d\to\infty}\frac{M_{d,\mu}}{N_{d,\mu}}\int_{\RR^2}f & =  \lim\limits_{d\to\infty} \frac{1}{N_{d,\mu}}\sum\limits_{S\in Cov_{d,\mu}}\int_{G/\Gamma}\hat{f}(gS) d\eta(g) \\
& = \int_{G/\Gamma}\Big(\lim\limits_{d\to\infty} \frac{1}{N_{d,\mu}} \sum\limits_{S\in Cov_{d,\mu}} \hat{f}(gS)\Big) d\eta(g)\nonumber \\
& =  \frac{1}{\nu(\HH_1(\mu))}\int_{G/\Gamma}d\eta(g)\int_{\HH_1(\mu)}\hat{f}d\nu(S). \nonumber \\
& =  \frac{1}{\nu(\HH_1(\mu))}\int_{\HH_1(\mu)}\hat{f}d\nu(S). \nonumber
\end{align}

The interchange of the limit and integral in (8) is valid due to the dominated convergence theorem in \cite{EM}. Now comparing (4) with (8), we obtain that 
$$ \lim\limits_{d\to\infty}\frac{M_{d,\mu}}{N_{d,\mu}} = c_{\mu}. $$

\section{Appendix B: Limits of slopes for $3\leq g\leq 6$}
Zorich wrote a program to calculate the Lyapunov exponents and Siegel-Veech constants for small genus. The corresponding table can be found in the appendix of \cite{EKZ}. 
Combining with Theorem~\ref{slc}, we can list the limit of slopes 
$$s_{\mu} = \lim\limits_{d\to\infty}s(\T)$$ 
in each stratum for small genus. In the following table, the first column consists of strata $\HH (\mu)$ or their connected components. The values of $s_{\mu}$ and their approximations are presented in the middle and the last columns, respectively. 

For $g=3$, we have
$$\begin{array}{lll}
\mathcal H^{hyp}(4) &
s_{\mu} =                 28/3 & 
s_{\mu}\approx            9.33333
 \\ 
\mathcal H^{odd}(4) &
s_{\mu} =                 9 & 
s_{\mu}\approx            9.
 \\ 
\mathcal H^{hyp}(2,2) &
s_{\mu} =                 28/3 & 
s_{\mu}\approx            9.33333
 \\ 
\mathcal H^{odd}(2,2) &
s_{\mu} =                 44/5 & 
s_{\mu}\approx            8.8
 \\ 
\mathcal H(3, 1) &
s_{\mu} =                 9 & 
s_{\mu}\approx            9.
 \\ 
\mathcal H(2, 1, 1) &
s_{\mu} =                 98/11 & 
s_{\mu}\approx            8.90909
 \\ 
\mathcal H(1, 1, 1, 1) &
s_{\mu} =                 468/53 & 
s_{\mu}\approx            8.83019
 \\ 
\end{array}
$$

For $g=4$, we have
$$\begin{array}{lll}
\mathcal H^{hyp}(6) &
s_{\mu} =                 9 &
s_{\mu}\approx            9.
 \\
\mathcal H^{even}(6) &
s_{\mu} =                 60/7 &
s_{\mu}\approx            8.57143
 \\
\mathcal H^{odd}(6) &
s_{\mu} =                 108/13 &
s_{\mu}\approx            8.30769
 \\
\mathcal H^{hyp}(3,3) &
s_{\mu} =                 9 &
s_{\mu}\approx            9.
 \\
\mathcal H^{nonhyp}(3,3) &
s_{\mu} =                 33/4 &
s_{\mu}\approx            8.25
 \\
\mathcal H^{even}(4, 2) &
s_{\mu} =                 17/2 &
s_{\mu}\approx            8.5
 \\
\mathcal H^{odd}(4, 2) &
s_{\mu} =                 236/29 &
s_{\mu}\approx            8.13793
 \\
\mathcal H^{even}(2, 2, 2) &
s_{\mu} =                 6180/737 &
s_{\mu}\approx            8.38535
 \\
\mathcal H^{odd}(2, 2, 2) &
s_{\mu} =                 8 &
s_{\mu}\approx            8.
 \\
\mathcal H(5, 1) &
s_{\mu} =                 25/3 &
s_{\mu}\approx            8.33333
 \\
\mathcal H(4, 1, 1) &
s_{\mu} =                 3118/379 &
s_{\mu}\approx            8.22691
 \\
\mathcal H(3, 2, 1) &
s_{\mu} =                 41/5 &
s_{\mu}\approx            8.2
 \\
\mathcal H(3, 1, 1, 1) &
s_{\mu} =                 65/8 &
s_{\mu}\approx            8.125
 \\
\mathcal H(2, 2, 1, 1) &
s_{\mu} =                 8178/1009 &
s_{\mu}\approx            8.10505
 \\
\mathcal H(2, 1, 1, 1, 1) &
s_{\mu} =                 1052/131 &
s_{\mu}\approx            8.03053
 \\
\mathcal H(1, 1, 1, 1, 1, 1) &
s_{\mu} =                 6675/839 &
s_{\mu}\approx            7.9559
 \\
\end{array}
$$

For $g=5$, we have 
$$\begin{array}{lll}
\mathcal H^{hyp}(8) &
s_{\mu} =                 44/5 & 
s_{\mu}\approx            8.8
 \\ 
\mathcal H^{even}(8) &
s_{\mu} =                 152064/18959 & 
s_{\mu}\approx            8.02068
 \\ 
\mathcal H^{odd}(8) & 
s_{\mu} =                 20004864/2574317 & 
s_{\mu}\approx            7.77094
 \\ 
\mathcal H^{hyp}(4,4) &
s_{\mu} =                 44/5 & 
s_{\mu}\approx            8.8
 \\ 
\mathcal H^{even}(4,4) &
s_{\mu} =                 70686/8977 & 
s_{\mu}\approx            7.87412
 \\ 
\mathcal H^{odd}(4,4) &
s_{\mu} =                 670629036/88307837 & 
s_{\mu}\approx            7.59422
 \\ 
\mathcal H^{even}(6, 2) &
s_{\mu} =                 1402948/178429 & 
s_{\mu}\approx            7.86278
 \\ 
\mathcal H^{odd}(6, 2) & 
s_{\mu} =                 176/23 & 
s_{\mu}\approx            7.65217
 \\ 
\mathcal H^{even}(4, 2, 2) & 
s_{\mu} =                 811372/104943 & 
s_{\mu}\approx            7.73155
 \\ 
\mathcal H^{odd}(4, 2, 2) &
s_{\mu} =                 2022416/269051 & 
s_{\mu}\approx            7.51685
 \\ 
\mathcal H^{even}(2, 2, 2, 2) &
s_{\mu} =                 998/131 & 
s_{\mu}\approx            7.61832
 \\ 
\mathcal H^{odd}(2, 2, 2, 2) &
s_{\mu} =                 2630076/355309 & 
s_{\mu}\approx            7.40222
 \\ 
\mathcal H(7, 1) &
s_{\mu} =                 18876/2423 & 
s_{\mu}\approx            7.79034
 \\ 
\mathcal H(6, 1, 1) &
s_{\mu} =                 456084222/59332837 & 
s_{\mu}\approx            7.68688
 \\ 
\mathcal H(5, 3) &
s_{\mu} =                 209/27 & 
s_{\mu}\approx            7.74074
 \\ 
\mathcal H(5, 2, 1) &
s_{\mu} =                 34386/4493 & 
s_{\mu}\approx            7.65324
 \\ 
\mathcal H(5, 1, 1, 1) & 
s_{\mu} =                 2344/309 & 
s_{\mu}\approx            7.58576
 \\ 
\mathcal H(4, 3, 1) &
s_{\mu} =                 3350523/438419 & 
s_{\mu}\approx            7.64229
 \\ 
\mathcal H(4, 2, 1, 1) &
s_{\mu} =                 4880938/646039 & 
s_{\mu}\approx            7.55518
 \\ 
\mathcal H(4, 1, 1, 1, 1) &
s_{\mu} =                 4797996/640763 & 
s_{\mu}\approx            7.48794
 \\ 
\mathcal H(3, 3, 2) &
s_{\mu} =                 466796/61307 & 
s_{\mu}\approx            7.61407
 \\ 
\mathcal H(3, 3, 1, 1) &
s_{\mu} =                 358044/47435 & 
s_{\mu}\approx            7.5481
 \\ 
\mathcal H(3, 2, 2, 1) &
s_{\mu} =                 45537/6049 & 
s_{\mu}\approx            7.52802
 \\ 
\mathcal H(3, 2, 1, 1, 1) &
s_{\mu} =                 20893/2800 & 
s_{\mu}\approx            7.46179
 \\ 
\mathcal H(3, 1, 1, 1, 1, 1) &
s_{\mu} =                 15537/2101 & 
s_{\mu}\approx            7.39505
 \\ 
\mathcal H(2, 2, 2, 1, 1) &
s_{\mu} =                 54114/7271 & 
s_{\mu}\approx            7.44244
 \\ 
\mathcal H(2, 2, 1, 1, 1, 1) &
s_{\mu} =                 1967748/266761 & 
s_{\mu}\approx            7.37645
 \\ 
\mathcal H(2, 1, 1, 1, 1, 1, 1) &
s_{\mu} =                 37451/5123 & 
s_{\mu}\approx            7.31037
 \\ 
\mathcal H(1, 1, 1, 1, 1, 1, 1, 1) &
s_{\mu} =                 569332/78587 & 
s_{\mu}\approx            7.24461
 \\ 
\end{array}
$$

For $g=6$, we have 
$$\begin{array}{lll}
\mathcal H^{hyp}(10) &
s_{\mu} =                 26/3 & 
s_{\mu}\approx            8.66667
 \\ 
\mathcal H^{even}(10) &
s_{\mu} =                 33950878311168/4542876976559 & 
s_{\mu}\approx            7.47343
 \\ 
\mathcal H^{odd}(10) & 
s_{\mu} =                 60850323456/8268054007 & 
s_{\mu}\approx            7.35969
 \\ 
\mathcal H^{hyp}(5,5) &
s_{\mu} =                 26/3 & 
s_{\mu}\approx            8.66667
 \\ 
\mathcal H^{nonhyp}(5,5) &
s_{\mu} =                 16222/2235 & 
s_{\mu}\approx            7.25817
 \\ 
\mathcal H^{even}(6, 4) &
s_{\mu} =                 1890839628/258411895 & 
s_{\mu}\approx            7.31715
 \\ 
\mathcal H^{odd}(6, 4) &
s_{\mu} =                 447260190/62068027 & 
s_{\mu}\approx            7.20597
 \\ 
\mathcal H^{even}(8, 2) &
s_{\mu} =                 5869052/799675 & 
s_{\mu}\approx            7.3393
 \\ 
\mathcal H^{odd}(8, 2) &
s_{\mu} =                 148352130/20523571 & 
s_{\mu}\approx            7.22838
 \\ 
\mathcal H^{even}(4, 4, 2) &
s_{\mu} =                 788452/109573 & 
s_{\mu}\approx            7.19568
 \\ 
\mathcal H^{odd}(4, 4, 2) &
s_{\mu} =                 1918814066/270844505 & 
s_{\mu}\approx            7.08456
 \\ 
\mathcal H^{even}(6, 2, 2) &
s_{\mu} =                 520529612/72176945 & 
s_{\mu}\approx            7.21185
 \\ 
\mathcal H^{odd}(6, 2, 2) &
s_{\mu} =                 4277366860/602374257 & 
s_{\mu}\approx            7.10085
 \\ 
\mathcal H^{even}(4, 2, 2, 2) &
s_{\mu} =                 7519897812/1060244863 & 
s_{\mu}\approx            7.0926
 \\ 
\mathcal H^{odd}(4, 2, 2, 2) &
s_{\mu} =                 15064684/2157689 & 
s_{\mu}\approx            6.98186
 \\ 
\mathcal H^{even}(2, 2, 2, 2, 2) &
s_{\mu} =                 758428/108479 & 
s_{\mu}\approx            6.99147
 \\ 
\mathcal H^{odd}(2, 2, 2, 2, 2) &
s_{\mu} =                 15297229/2223051 & 
s_{\mu}\approx            6.88119
 \\ 
\mathcal H(9, 1) &
s_{\mu} =                 37708398/5152405 & 
s_{\mu}\approx            7.3186
 \\ 
\mathcal H(8, 1, 1) &
s_{\mu} =                 1863320550/257996231 & 
s_{\mu}\approx            7.22228
 \\ 
\mathcal H(7, 3) &
s_{\mu} =                 462292740/63606061 & 
s_{\mu}\approx            7.26806
 \\ 
\mathcal H(7, 2, 1) &
s_{\mu} =                 38090780/5298767 & 
s_{\mu}\approx            7.18861
 \\ 
\mathcal H(7, 1, 1, 1) & 
s_{\mu} =                 1209740/169723 & 
s_{\mu}\approx            7.12773
 \\ 
\mathcal H(6, 3, 1) &
s_{\mu} =                 3587177295/499975363 & 
s_{\mu}\approx            7.17471
 \\ 
\mathcal H(6, 2, 1, 1) &
s_{\mu} =                 3706261910/522310211 & 
s_{\mu}\approx            7.0959
 \\ 
\mathcal H(6, 1, 1, 1, 1) &
s_{\mu} =                 4604674588/654501283 & 
s_{\mu}\approx            7.03539
 \\ 
\mathcal H(5, 4, 1) &
s_{\mu} =                 10603/1479 & 
s_{\mu}\approx            7.16903
 \\ 
\mathcal H(5, 3, 2) &
s_{\mu} =                 77025/10783 & 
s_{\mu}\approx            7.14319
 \\ 
\mathcal H(5, 3, 1, 1) &
s_{\mu} =                 119875/16923 & 
s_{\mu}\approx            7.08355
 \\ 
\mathcal H(5, 2, 2, 1) &
s_{\mu} =                 6824055/965852 & 
s_{\mu}\approx            7.06532
 \\ 
\mathcal H(5, 2, 1, 1, 1) &
s_{\mu} =                 20764920/2964019 & 
s_{\mu}\approx            7.00566
 \\ 
\mathcal H(5, 1, 1, 1, 1, 1) &
s_{\mu} =                 105124/15135 & 
s_{\mu}\approx            6.94575
 \\ 
\mathcal H(4, 4, 1, 1) &
s_{\mu} =                 969884818/136989269 & 
s_{\mu}\approx            7.08001
 \\ 
\mathcal H(4, 3, 3) &
s_{\mu} =                 1175124132/164753321 & 
s_{\mu}\approx            7.13263
 \\ 
\mathcal H(4, 3, 2, 1) &
s_{\mu} =                 37937237/5377361 & 
s_{\mu}\approx            7.05499
 \\ 
\mathcal H(4, 3, 1, 1, 1) &
s_{\mu} =                 101001003/14437984 & 
s_{\mu}\approx            6.99551
 \\ 
\mathcal H(4, 2, 2, 1, 1) &
s_{\mu} =                 5749826838/824007139 & 
s_{\mu}\approx            6.97788
 \\ 
\mathcal H(4, 2, 1, 1, 1, 1) &
s_{\mu} =                 6549795692/946682951 & 
s_{\mu}\approx            6.91868
 \\ 
\mathcal H(4, 1, 1, 1, 1, 1, 1) &
s_{\mu} =                 425743929/62066287 & 
s_{\mu}\approx            6.8595
 \\ 
\mathcal H(3, 3, 3, 1) &
s_{\mu} =                 858690/121831 & 
s_{\mu}\approx            7.04821
 \\ 
\mathcal H(3, 3, 2, 2) &
s_{\mu} =                 79129020/11255951 & 
s_{\mu}\approx            7.02997
 \\ 
\mathcal H(3, 3, 2, 1, 1) &
s_{\mu} =                 70230580/10074361 & 
s_{\mu}\approx            6.97122
 \\ 
\mathcal H(3, 3, 1, 1, 1, 1) &
s_{\mu} =                 479973180/69439639 & 
s_{\mu}\approx            6.91209
 \\ 
\mathcal H(3, 2, 2, 2, 1) &
s_{\mu} =                 53539095/7699507 & 
s_{\mu}\approx            6.95357
 \\ 
\mathcal H(3, 2, 2, 1, 1, 1) &
s_{\mu} =                 60991725/8845784 & 
s_{\mu}\approx            6.89501
 \\ 
\mathcal H(3, 2, 1, 1, 1, 1, 1) &
s_{\mu} =                 9250555/1353139 & 
s_{\mu}\approx            6.83637
 \\ 
\mathcal H(3, 1, 1, 1, 1, 1, 1, 1) &
s_{\mu} =                 178679025/26361772 & 
s_{\mu}\approx            6.77796
 \\ 
\mathcal H(2, 2, 2, 2, 1, 1) &
s_{\mu} =                 1157266990/168259347 & 
s_{\mu}\approx            6.87788
 \\ 
\mathcal H(2, 2, 2, 1, 1, 1, 1) &
s_{\mu} =                 43884132/6434891 & 
s_{\mu}\approx            6.81972
 \\ 
\mathcal H(2, 2, 1, 1, 1, 1, 1, 1) & 
s_{\mu} =                 1130218035/167149999 & 
s_{\mu}\approx            6.7617
 \\ 
\mathcal H(2, 1, 1, 1, 1, 1, 1, 1, 1) &
s_{\mu} =                 85529164/12757861 & 
s_{\mu}\approx            6.70404
 \\ 
\mathcal H(1, 1, 1, 1, 1, 1, 1, 1, 1, 1) &
s_{\mu} =                 122875578/18486283 & 
s_{\mu}\approx            6.64685
 \\ 
\end{array}
$$


\begin{thebibliography}{[A9]}

\bibitem[BDPP]{BDPP}
S. Boucksom, J.-P. Demailly, M. Paun and T. Peternell, \textit{
The pseudo-effective cone of a compact K\"{a}hler manifold and varieties of negative Kodaira dimension}, arXiv:math/0405285.

\bibitem[BM]{BM}
I. Bouw and M. M\"{o}ller, \textit{Teichm\"{u}ller curves, triangle groups, and Lyapunov exponents}, Ann. of Math., to appear. 

\bibitem[CT]{CT}
A.-M. Castravet and J. Tevelev, \textit{Exceptional Loci on $\Mon$ and Hypergraph Curves}, arXiv:0809.1699. 

\bibitem[C]{C}
D. Chen, \textit{Covers of elliptic curves and the moduli space of stable curves}, J. Reine Angew. Math., to appear.

\bibitem[EKZ]{EKZ}
A. Eskin, M. Kontsevich and A. Zorich, \textit{Sum of Lyapunov exponents of the Hodge bundle with respect to the Teichm\"{u}ller geodesic flow}, preprint, 2009. 

\bibitem[EM]{EM}
A. Eskin and H. Masur, \textit{Pointwise asymptotic formulas on flat surfaces}, Ergodic Theory Dynam. Systems 21 (2001), no. 2, 443--478.

\bibitem[EMZ]{EMZ}
A. Eskin, H. Masur and A. Zorich, \textit{Moduli spaces of abelian differentials: the principal boundary, counting problems, and the Siegel-Veech constants}, Publ. Math. Inst. Hautes \'{E}tudes Sci. No. 97 (2003), 61--179. 

\bibitem[EO]{EO}
A. Eskin and A. Okounkov, \textit{Asymptotic of numbers of
branched coverings of a torus and volumes of moduli spaces of
holomorphic differentials}, Invent. Math. 145 (2001), 59--103.

\bibitem[G]{G}
A. Gibney, \textit{Numerical criteria for divisors on $\Mg$ to be ample}, Compos. Math., to appear. 

\bibitem[Gr]{Gr}
S. Grushevsky, \textit{Geometry of $\mathcal A_g$ and its compactifications}, Algebraic geometry--Seattle 2005. Part 1, 193--234, Proc. Sympos. Pure Math., 80, Part 1, Amer. Math. Soc., Providence, RI, 2009.

\bibitem[HM1]{HM1}
J. Harris and I. Morrison, \textit{Slopes of effective divisors on the moduli space of stable
curves}, Invent. Math. 99 (1990), 321--355.

\bibitem[HM2]{HM2}
J. Harris and I. Morrison, \textit{Moduli of curves}, Springer-Verlag New York, 1998.

\bibitem[H]{H}
B. Hassett, \textit{Moduli spaces of weighted pointed stable curves}, Adv. Math. 173  (2003), no. 2, 316--352.

\bibitem[HS]{HS}
P. Hubert and T. Schmidt, \textit{An Introduction to Veech Surfaces}, Handbook of dynamical systems, Vol. 1B,  501--526, Elsevier B. V., Amsterdam, 2006.

\bibitem[KM]{KM}
S. Keel and J. McKernan, \textit{Contractible Extremal Rays on $\Mon$}, arXiv:alg-geom/9707016. 

\bibitem[K]{K}
M. Kontsevich, \textit{Lyapunov exponents and Hodge theory}, The mathematical beauty of physics (Saclay, 1996), 318--332,
Adv. Ser. Math. Phys., 24, World Sci. Publ., River Edge, NJ, 1997. 

\bibitem[KZ]{KZ}
M. Kontsevich and A. Zorich, \textit{Connected components of the moduli spaces of Abelian differentials with prescribed singularities}, 
Invent. Math. 153 (2003), no. 3, 631--678. 

\bibitem[CTM]{CTM}
C.T. McMullen, \textit{Rigidity of Teichm\"{u}ller curves}, Math. Res. Lett. 16 (2009), no. 4, 647--649. 

\bibitem[M]{M}
M. M\"{o}ller, \textit{Variations of Hodge structures of a TeichmŸller curve}, J. Amer. Math. Soc. 19 (2006), no. 2, 327--344. 

\bibitem[Mo]{Mo}
I. Morrison, \textit{Mori theory of moduli spaces of stable curves}, preprint, 2009. 

\bibitem[P]{P}
R. Pandharipande, \textit{Descendent bounds for effective divisors on the moduli space of curves}, arXiv:0805.0601. 

\bibitem[V1]{V1}
W.A. Veech, \textit{Geometric realizations of hyperelliptic curves}, Algorithms, fractals, and dynamics (Okayama/Kyoto, 1992), 217--226, Plenum, New York, 1995. 

\bibitem[V2]{V2}
W.A. Veech, \textit{Siegel measures}, Ann. of Math. (2) 148 (1998), no. 3, 895--944. 

\end{thebibliography}
\end{document}